\documentclass[onefignum,onetabnum]{siamart220329}

\usepackage[utf8]{inputenc}

\usepackage{amsmath}
\usepackage{amssymb}
\usepackage{pifont}
\usepackage{braket}
\usepackage{etoolbox}
\usepackage{upgreek}
\usepackage{comment}
\usepackage{stmaryrd}
\usepackage{appendix}
\usepackage{bbm}
\usepackage{url}
\usepackage{dsfont}
\usepackage{tikz-cd}
\usepackage{adjustbox}
\usepackage[shortlabels]{enumitem}

\makeatletter
\patchcmd{\@addmarginpar}{\ifodd\c@page}{\ifodd\c@page\@tempcnta\m@ne}{}{}
\makeatother
\reversemarginpar

\renewcommand{\marginpar}[2][]{}

\newtheorem{example}[theorem]{Example}
\newtheorem{remark}[theorem]{Remark}
\newtheorem{result}[theorem]{Result}
  
\newcommand{\generate}[2]{\langle #1 \rangle_{\mathsf{#2}}}

\newcommand{\Z}{\mathbb{Z}}
\newcommand{\R}{\mathbb{R}}
\newcommand{\C}{\mathbb{C}}

\newcommand{\D}{\mathrm{D}}

\newcommand{\cL}{\mathcal{L}}

\newcommand{\cP}{\mathcal{P}}

\newcommand{\fa}{\mathfrak{a}}

\newcommand{\SU}{{\rm SU}}
\newcommand{\SO}{{\rm SO}}
\newcommand{\GL}{{\rm GL}}
\newcommand{\SL}{{\rm SL}}

\newcommand{\so}{\mathfrak{so}}

\newcommand{\mf}[1]{\mathfrak{#1}}
\newcommand{\mb}[1]{\mathbf{#1}}
\newcommand{\mc}[1]{\mathcal{#1}}

\newcommand{\down}{\shortdownarrow}
\newcommand{\myparallel}{{\mkern3mu\vphantom{\perp}\vrule depth 0pt\mkern2mu\vrule depth 0pt\mkern3mu}}
\renewcommand{\parallel}{\myparallel}

\newcommand{\todo}[1]{
\textcolor{blue}{TODO: #1}
}

\newcommand{\Int}{{\rm Int}}
\renewcommand{\ker}{{\rm ker}}
\newcommand{\img}{{\rm im}}

\newcommand{\id}{{\mathds1}}
\newcommand{\conv}{{\rm conv}}
\newcommand{\linspan}{{\rm span}}

\newcommand{\interior}{{\rm int}}

\newcommand{\Ad}{{\rm Ad}}
\newcommand{\ad}{{\rm ad}}

\newcommand{\stab}{\mathsf{stab}}
\newcommand{\reach}{\mathsf{reach}}
\newcommand{\derv}{\mathsf{derv}}
\newcommand{\sol}{\mathsf{sols}}

\renewcommand{\epsilon}{\varepsilon}

\headers{Reduced Control Systems on Symmetric Lie Algebras}{E.~Malvetti, G.~Dirr, F.~vom Ende, and T.~Schulte-Herbr\"uggen}

\title{Reduced Control Systems on Symmetric Lie Algebras}

\author{
Emanuel Malvetti\thanks{School of Natural Sciences, Technische Universit\"at M\"unchen, 85737 Garching, Germany, and Munich Centre for Quantum Science and Technology (MCQST) \& Munich Quantum Valley (MQV), 80799 M{\"u}nchen, Germany}
\and 
Gunther Dirr\thanks{Department of Mathematics, University of W{\"u}rzburg, 97074 W{\"u}rzburg, Germany}
\and 
Frederik vom Ende\thanks{Dahlem Center for Complex Quantum Systems, Freie Universit{\"a}t Berlin, 14195 Berlin, Germany} 
\and
\mbox{Thomas Schulte-Herbr\"uggen\footnotemark[1]}}

\begin{document}


\maketitle
\centerline{\today}

\begin{abstract}
For a symmetric Lie algebra $\mf g=\mf k\oplus\mf p$ we consider a class of bilinear or more general control-affine systems on $\mf p$ defined by a drift vector field $X$ and control vector fields $\ad_{k_i}$ for $k_i\in\mf k$ such that one has
fast and full control on 
the corresponding compact group $\mb K$. We show that under quite general assumptions on $X$ such a control system is essentially equivalent to a natural reduced system on a maximal Abelian subspace $\mf a\subseteq\mf p$, and likewise to related differential inclusions defined on $\mf a$. We derive a number of general results for such systems and as an application we prove a simulation result with respect to the preorder induced by the Weyl group action. 
\end{abstract}

\begin{keywords}
Reduced control system, control-affine system, symmetric Lie algebra
\end{keywords}

\begin{MSCcodes}
37N20, 
93B03, 
93B05, 
93D99, 
93C10 
\end{MSCcodes}

\section{Introduction}

\marginpar{Can/should we cite the extended abstract:~\cite{MTNS_Malvetti}?}

\subsection{Motivation}

We consider control systems that admit fast controllability on certain degrees of freedom represented by a Lie group action. 
Intuitively, one should be able to factor out these degrees of freedom, and so our goal is to define an associated reduced control system on the remaining degrees of freedom, and to show that the two systems are essentially equivalent, in a sense which will be specified later.

This idea has been considered in~\cite[Ch.~22]{Agrachev04} for commuting controls under the assumption that the reduced state space is again a manifold. 
In our setting the controls do not commute and the reduced state space has singularities, which are the source of most complications.
The idea of considering a reduced state space---even if the reduced control system is not defined explicitly---has come up several times in quantum control theory. If the reduced state space is a Riemannian symmetric space, strong results can be derived~\cite{Khaneja01b,Burgarth23}. Unfortunately such systems are rare in practice. 
Often the quotient spaces are rather complicated, and one contents oneself with finding diameters of such spaces to derive speed limits~\cite{Gauthier21}.
Our paper will generalize the ideas presented in~\cite{Sklarz04,rooney2018,CDC19} in a mathematically rigorous manner.

\marginpar{Make notation consistent with worked example}
We give a simple example to motivate our work. Consider the closed unit disk $D\subset\R^2$ in the plane and let $X$ be some complete and sufficiently smooth vector field on $D$, such that $D$ is invariant under the flow of $X$. The compact Lie group $\SO(2)$ acts on the disk by rotations. Now consider a control system on $D$ with constant drift $X$ and fast control on the action of $\SO(2)$. Without the drift term, this means that we can move arbitrarily quickly within the orbits of the group action, which in this case are simply the concentric circles about the origin. Including the drift term this is still approximately true. Hence points on the same orbit may be considered equivalent, and the question becomes how one can move between orbits. This suggests that there should be a natural way to define a corresponding control system on the quotient space $D/\SO(2)\cong[0,1]$, which in our example is the set of all radii\footnote{Note that the two boundary points of the quotient space have a different meaning. Here $1$ comes from the boundary of the disk, whereas $0$ originates from the singular $\SO(2)$-orbit. This is important for defining the appropriate notion of differentiability in the quotient space, see~\cite[App.~B]{diag}.}. Moreover, we want this reduced control system to be equivalent to the original system in some precise sense, so that no information is lost.

Let us see what this reduced control system should look like in our simple example. Instead of working on the quotient space, which in general is not a manifold, we will look at a subspace of our state space which intersects all orbits a finite number of times, and does so orthogonally. Here we choose the intersection of the horizontal axis with the disk, i.e.~the line segment $A=[-1,1]\cdot e_1\subset\R^2$ where $e_1=(1,0)$. This will be our new reduced state space. If we restrict the drift vector field $X$ to the axis $A$ and project the vector field orthogonally onto the axis, this yields some possible dynamics on the reduced space.
Using the fast control we can rotate our horizontal axis $A$ to any other axis, and obtain a different vector field on the reduced space. Collecting all of these vector fields defines the reduced control system.

We can plot these vector fields all together in a single graph, where the abscissa is the reduced state space, see Figure~\ref{fig:motivation}. In the example $X$ is affine linear, and so are the restricted vector fields and hence the graph is a collection of lines. This can be seen as a set-valued function, and the reduced control system can be seen as the corresponding differential inclusion, as we will show below.

\begin{figure}
\centering
\includegraphics[width=0.35\textwidth]{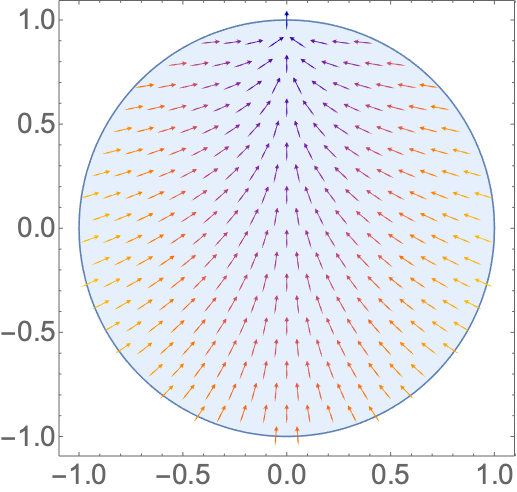}
\hspace{5mm}
\includegraphics[width=0.55\textwidth]{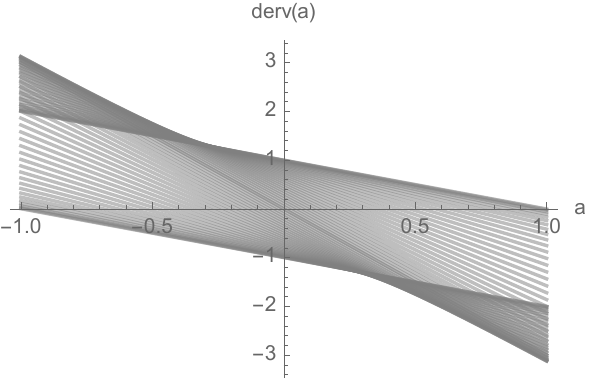}
\caption{(Color online) Left: A drift vector field $X$ on the disk $D$. Right: A plot of the corresponding differential inclusion defined on the interval $[-1,1]\cong A$.}
\label{fig:motivation}
\end{figure}

We will come back to this example in Section~\ref{sec:worked-example} where we use it to illustrate how the method of reduced control systems can be used in practice.

\subsection{Outline}

A concise introduction to symmetric Lie algebras is provided in Section~\ref{sec:sym-lie-alg} and made concrete with some well-known matrix examples.

In Section~\ref{sec:control-systems} we introduce the control systems studied in this paper, in particular we define the reduced control system in Section~\ref{sec:reduced-systems}. 
Moreover, in Sections~\ref{sec:op-lift-orig} and~\ref{sec:op-lift-red} we briefly address the operator lift of the original and reduced control systems. 
Some basic properties of the reduced control system are collected in Appendix~\ref{app:ctrl-sys-props}.

We then go on to prove our main results in Section~\ref{sec:equivalence}, establishing the equivalence of the reduced control system and the original one. 
We start out with a local equivalence result in Proposition~\ref{prop:local-equivalence}, followed by the global equivalence result, which will be separated into a projection, see Theorem~\ref{thm:projection-equivalence}, and a lift, see Theorem~\ref{thm:lift-equivalence}. 

In Section~\ref{sec:consequences} we explore how statements about important control theoretic concepts such as reachability, viability, controllability and accessibility can be determined using the reduced control system and lifted to the original one.

As an application, in Section~\ref{sec:simulation} we give a simulation result which establishes the preorder induced by the Weyl group action as a kind of resource, see Theorem~\ref{thm:simulation}.

Finally, the motivational example given above is worked out in detail in Section~\ref{sec:worked-example}.

\subsection{Symmetric Lie Algebras} \label{sec:sym-lie-alg}

The control systems studied in this paper are defined using symmetric Lie algebras. 
Although this setting might seem abstract, it is actually quite familiar as it generalizes common matrix diagonalizations, such as eigenvalue and singular value decompositions, as shown in the examples below.
In fact the results of this paper should be understandable without any prior knowledge of symmetric Lie algebras, and we recommend that the reader interprets the results using one of the concrete examples.

In order to introduce notation we give a very concise summary here. 
We will frequently use results from our previous work~\cite{diag}, and for a thorough introduction to symmetric Lie algebras we refer to Appendix A of said paper. 

A symmetric Lie algebra is a (real, finite dimensional) Lie algebra $\mf g$ together with an involutive Lie algebra automorphism $s$. 
This yields a vector space decomposition $\mf g=\mf k\oplus\mf p$ into $+1$ and $-1$ eigenspaces of $s$ which we call Cartan-like decomposition since it generalizes the usual Cartan decomposition. 
Importantly we have the following commutator relations:
$[\mf k,\mf k]\subseteq\mf k,\,[\mf k,\mf p]\subseteq\mf p,\,[\mf p,\mf p]\subseteq\mf k$.
Given a Lie group $\mb G$ with Lie algebra $\mf g$, let $\mb K\subseteq\mb G$ be the analytic subgroup generated by $\mf k$. We say that the pair $(\mb G,\mb K)$ is associated to the symmetric Lie algebra. One can show that the adjoint action of $\mb K$ on $\mf g$ leaves $\mf p$ invariant.
The corresponding quotient map is denoted $\pi:\mf p\to\mf p/\mb K$.
We will only consider symmetric Lie algebras which are semisimple and orthogonal. In particular there exists an inner product on $\mf g$ which is invariant under $s$ and $\ad_{\mf k}$, and this implies that the group $\Ad_{\mb K}$ is compact.
If $\mf a\subseteq\mf p$ is a maximal Abelian subspace, then every point $x\in\mf p$ can be mapped to $\mf a$ by some $K\in\mb K$, that is, $\Ad_K(x)\in\mf a$. This generalizes the idea of diagonalization. However, the resulting element $\Ad_K(x)\in\mf a$ is not unique, since the elements of $\mb K$ which leave $\mf a$ invariant can act non-trivially on $\mf a$. The resulting group of transformations of $\mf a$ is called the Weyl group, denoted $\mb W$, and it is a finite group generated by reflections. A convenient fact about Weyl groups is that they admit a (closed) Weyl chamber $\mf w\subseteq\mf a$, such that each orbit $\Ad_{\mb K}(x)$ intersects $\mf w$ in exactly one point\footnote{Put differently, and more generally, the Weyl chamber $\mf w$ and the quotients $\mf p/\mb K$ and $\mf a/\mb W$ are isometrically isomorphic.}. If this point lies in the relative interior of $\mf w$ (w.r.t.~$\mf a$), then $x$ is called regular. Note that even if we fix $\Ad_K(x)\in\mf a$, the element $K\in\mb K$ is still not unique. 

Orthogonality of the symmetric Lie algebra has some important geometric consequences. Let $x\in\mf p$ and consider
$\mf p_x:=\{y\in\mf p : [x,y]=0\}$, i.e.~the commutant of $x$ in $\mf p$. A key fact is that the orbit of $\mb K$ through $x$, denoted by $\mb Kx$, is orthogonal to $\mf p_x$ at $x$. Since the tangent space of the orbit at $x$ can be identified with $\ad_{\mf k}(x)$, we can define the orthogonal projection $\Pi_x:\mf p\to\mf p$ with image $\mf p_x$ and kernel $\ad_{\mf k}(x)$, yielding the useful equation
\begin{align} \label{eq:orth-orbs}
\Pi_x(\ad_k(x))=0 \quad \text{ for all } x\in\mf p,\,k\in\mf k.
\end{align}

Now let us give some examples relating certain matrix diagonalizations to symmetric Lie algebras.
In fact many common matrix diagonalizations---such as eigenvalue and singular value decompositions---and also some uncommon ones, can be rephrased in the setting of symmetric Lie algebras~\cite{Kleinsteuber}.

\begin{example}[Hermitian EVD] \label{ex:hermitian-evd}
Consider the pair $(\SL(n,\C), \SU(n))$ which is associated to the semisimple orthogonal symmetric Lie algebra $\mf{sl}(n,\C)=\mf{su}(n)\oplus\mf{herm}_0(n,\C)$. 
The adjoint action of $X\in\SL(n,\C)$ on $Y\in\mf{sl}(n,\C)$ is given by conjugation, that is, $\Ad_X(Y)=XYX^{-1}$ and similarly for $X,Y\in\mf{sl}(n,\C)$ it holds that $\ad_X(Y)=[X,Y]=XY-YX$.
The automorphism here is $s(X)=-X^*$. 
A convenient choice of a maximal Abelian subspace of $\mf{herm}_0(n,\C)$ is the subset of all diagonal matrices. 
These will automatically be real and traceless. 
We denote this set by $\mf d_0(n,\R)$. 
Hence the symmetric Lie algebra encapsulates the idea of unitary diagonalization of Hermitian matrices.
The corresponding Weyl group is isomorphic to the symmetric group $S_n$ acting on $n$ elements. The action on $\mf d_0(n,\R)$ is given by permutation of the diagonal elements of the matrix. A natural choice of a Weyl chamber is the subset of $\mf d_0(n,\R)$ with the diagonal elements in non-increasing order. 
\end{example}

\begin{example}[Real SVD] \label{ex:real-svd}
The real singular value decomposition also corresponds to a symmetric Lie algebra, although the connection is less obvious than in Example~\ref{ex:hermitian-evd}.
The pair $(\SO(p,q), \SO(p)\times\SO(q))$ is associated to the semisimple orthogonal symmetric Lie algebra $\mf{so}(p,q)$ with $\mf k=\so(p)\oplus\so(q)$ and $\mf p$ equal to the set of matrices of the form $\left(\begin{smallmatrix}0&B\\B^\top&0\end{smallmatrix}\right)$ where $B\in\R^{p\times q}$. 
A maximal Abelian subspace is given by such matrices with $B$ diagonal, and the Weyl group acts by permutations and sign flips, so it is isomorphic to the signed symmetric group $\Z_2\wr S_{p\wedge q}$ (here $\wr$ denotes the wreath product). 
The Weyl chamber consists of all diagonal matrices with non-negative diagonal elements in non-increasing order.
The connection to the SVD stems from the adjoint action which is $\Ad_{(V,W)}\left(\begin{smallmatrix}0&B\\B^\top&0\end{smallmatrix}\right)=\left(\begin{smallmatrix}0&VBW^\top\\(VBW^\top)^\top&0\end{smallmatrix}\right)$.
\end{example}

As a special case we obtain the following, which also covers the motivational example given above.

\begin{example}[Polar decomposition of $\R^n$] \label{ex:polar-dec}
Choosing $p=n$ and $q=1$ in Example~\ref{ex:real-svd} yields the polar decomposition of $\R^n$, meaning that $\mf p\cong\R^n$ and $\mf k\cong\mf{so}(n)$. The maximal Abelian subspaces are exactly the lines through the origin, with the Weyl group being isomorphic to $\Z_2$.
\end{example}

\section{Control Systems} \label{sec:control-systems}

For the remainder of the paper we will be working with a semisimple orthogonal symmetric Lie algebra $(\mf g,s)$ with Cartan-like decomposition $\mf g=\mf k\oplus\mf p$ and an associated pair $(\mb G,\mb K)$ with $\mb K$ compact and connected\footnote{This is always possible, for instance by choosing $\mb G=\Int(\mf g)$ and $\mb K=\Int_{\mf k}(\mf g)$, see~\cite[Lem.~A.20]{diag}.}. Moreover $\mf a\subseteq\mf p$ denotes some choice of a maximal Abelian subspace, with Weyl group $\mb W$ and a closed Weyl chamber $\mf w$.

We start by defining the class of control-affine systems on $\mf p$ that we want to study in the sequel. 
We are given a vector field $X$ on $\mf p$, called the \emph{drift vector field}, and a set of \emph{control directions} $k_1,\ldots,k_m\in\mf k$.
The control system we wish to study in this work is the following:
\begin{align} \label{eq:control-affine}
p'(t)=X(p(t)) + \sum_{i=1}^m u_i(t)\ad_{k_i}(p(t)), \quad p(0)=p_0\in\mf p    
\tag{\sf A}
\end{align}
where $\ad_x$ denotes the adjoint operator of $x$, that is, $\ad_x(y):=[x,y]$. 
We will always consider solutions on an interval $I$ of the form $[0,T]$ with $T\geq0$, or of the form $[0,\infty)$.
The \emph{control functions} $u_i:I\to\R$ are required to be locally integrable, 
refer to \cite[App.~C]{Sontag98}.
A solution $p:I\to\mf p$ is an absolutely continuous function satisfying~\eqref{eq:control-affine} almost everywhere for some choice of control functions. 
Of course when $X$ is linear then~\eqref{eq:control-affine} is in fact a bilinear control system~\cite{Elliott09}.

The two key assumptions made throughout this paper are:
\begin{enumerate}[(I)] 
\item \label{it:full-ctrl} The control directions generate the full Lie algebra: $\generate{k_1,\ldots,k_m}{LA}=\mf k$.
\item \label{it:fast-ctrl} The control functions $u_i:I\to\R$ may be unbounded, as they are only required to be locally integrable.
\end{enumerate}

\noindent Under these assumptions, and if we neglect the drift $X$, we can move between any two points of a given $\mb K$-orbit in $\mf p$ arbitrarily quickly, cf.~\cite[Prop.~2.7]{Elliott09}. We say that we have \emph{fast and full control} on the Lie group $\mb K$---and thus on its orbits in $\mf p$.
Some results will use the following strengthened version of~\ref{it:full-ctrl}:
\begin{enumerate}[(I')] 
\item \label{it:full-ctrl-2} The control directions span the full Lie algebra: $\linspan({k_1,\ldots,k_m})=\mf k$.
\end{enumerate}

\subsection{Reduced Control System} \label{sec:reduced-systems}

Assumptions~\ref{it:full-ctrl} \&~\ref{it:fast-ctrl} imply that we can move into the maximal Abelian subspace $\mf a$ at any time. This motivates us to define a reduced control system on $\mf a$. First we introduce some concepts. For every $K\in \mb K$, we define the \emph{induced vector field} on $\mf a$
\begin{align*}
X_K:=\Pi_{\mf a}\circ\Ad_K^\star(X)\circ\,\iota\,,
\end{align*}
where $\Pi_{\mf a}:\mf p\to\mf a$ is the orthogonal projection on $\mf a$, and $\iota:\mf a\hookrightarrow\mf p$ is the inclusion\footnote{In the following we will usually suppress the inclusion $\iota$ from the notation.}. By $\Ad_K$ we denote the adjoint action of $K$ on $\mf p$ and $\Ad_K^\star$ denotes the pullback action, that is, $\Ad_K^\star(X)=\Ad_K^{-1}\circ X\circ\Ad_K$. If $X$ is linear, then so are all $X_K$. We denote by $\mf X:=\{X_K:K\in\mb K\}$ the set of induced vector fields.

Now we can define the \emph{reduced control system} by
\begin{align} \label{eq:reduced}
a'(t)=X_{K(t)}(a(t)), \quad a(0)=a_0\in\mf a\,,
\tag{\sf R}
\end{align}
where the control function $K:I\to \mb K$ is required to be measurable. Again, a solution is an absolutely continuous function $a:I\to\mf a$ which satisfies~\eqref{eq:reduced} almost everywhere. 

Moreover, we define the set of achievable derivatives at $a\in\mf a$ by
\begin{align*}
\derv(a)=\{X_K(a):K\in\mb K\}\subset T_a\mf a\cong\mf a\,.
\end{align*}
Then we can also define a differential inclusion corresponding to~\eqref{eq:reduced} by
\begin{align} \label{eq:inclusion}
a'(t)\in \derv(a(t)), \quad a(0)=a_0\in\mf a\,,
\tag{\sf I}
\end{align}
where $a:I\to\mf a$ needs to be absolutely continuous and satisfy~\eqref{eq:inclusion} almost everywhere. In fact~\eqref{eq:reduced} and~\eqref{eq:inclusion} are equivalent, i.e.~they have the same solutions. This follows from Filippov's theorem, cf.~\cite[Thm.~2.3]{Smirnov02}.
The difference between~\eqref{eq:reduced} and~\eqref{eq:inclusion} is that the latter ``forgets'' about the controls, and leads to a more static, geometric picture. We will switch between both viewpoints whenever it simplifies things.

Often it will be convenient to consider a relaxed version of the differential inclusion above given by 
\begin{align} \label{eq:relaxed}
a'(t)\in \conv(\derv(a(t))), \quad a(0)=a_0\in\mf a\,,
\tag{\sf C}
\end{align}
where $\conv$ denotes the convex hull. This will slightly enlarge the set of solutions,
however, every solution of~\eqref{eq:relaxed} can still be approximated uniformly on compact time intervals by solutions to~\eqref{eq:inclusion}, see~\cite[Ch.~2.4, Thm.~2]{Aubin84}.

\subsection{Operator Lift of the Original Control System} \label{sec:op-lift-orig} 

Both the original and the reduced control system can be lifted to the operator level.
We collect some results here as they are of general interest, but they will not be important for the remainder of this paper.
For this section we use Assumption~\ref{it:fast-ctrl}, but we drop Assumption~\ref{it:full-ctrl}.
Let $X\in\mf{gl}(\mf p)$ be a linear\footnote{
If one wants to consider more general vector fields, the operator lift of the control-affine system will typically be defined on  
an infinite-dimensional space.} 
vector field on $\mf p$ and consider the operator lift of~\eqref{eq:control-affine} on $\GL(\mf p)$ given by the following bilinear system:
\begin{align} \label{eq:op-lift}
L'=\Big(X+\sum_{i=1}^m u_i(t)\ad_{k_i}\Big)L\,, \quad L(0)=\id\in\GL(\mf p)\,. \tag{\sf OP}
\end{align}

Such right-invariant control systems defined on Lie groups are highly structured and allow for the application of Lie semigroup theory, see~\cite{Lawson99} for a concise introduction. This system is characterized by the set $\Omega:=\{X+\ad_k:k\in\linspan(k_1,\ldots,k_m)\}$. The fact that $\Omega$ is not bounded causes some problems but will be remedied below by passing to the reduced control system. We write $\generate{\Omega}{wedge}, \generate{\Omega}{LW}, \generate{\Omega}{LS}, \generate{\Omega}{LA}$ for the wedge, Lie wedge, Lie saturate, and Lie algebra generated by $\Omega$, respectively. Again see~\cite{Lawson99} for the definitions and main results.

\begin{proposition} \label{prop:lift-wedges}
Let $\mf h=\generate{k_1,\ldots,k_m}{LA}$ and let $\mb H\subseteq\mb K$ be the corresponding analytic subgroup. 
The following statements hold.\smallskip
\begin{enumerate}[(i)]
\item \label{it:wedges:W} $\generate{\Omega}{wedge}=\R_+ X+\linspan(\ad_{k_1},\ldots,\ad_{k_m})$ and $\generate{\Omega}{LW}\supseteq\generate{\ad_{\mf h},\Ad_{\mb H}^\star(X)}{wedge}$.
\item \label{it:wedges:LW} Assume that there is a Lie wedge $\mf v$ such that $\ad_{\mf h}\subseteq E(\mf v)$ and such that $X\in \mf v\setminus E(\mf v)$.
Then $\generate{\Omega}{LW}=\generate{\ad_{\mf h},\Ad_{\mb H}^\star(X)}{wedge}$.
\item \label{it:wedges:LS} If, in addition, $\ad_{\mf h}$ and $\mf v$ are global\,\footnote{Recall that this means that the Lie algebra $\ad_{\mf h}$ generates a closed Lie subgroup in $\GL(\mf p)$.}, then $\generate{\Omega}{LS}=\generate{\ad_{\mf h},\Ad_{\mb H}^\star(X)}{wedge}$.
\item \label{it:wedges:LA} If $X,\ad_{h_i}\subseteq\mf l$ for some compact Lie algebra\footnote{By this we mean that the Lie group generated by $\mf l$ in $\GL(\mf p)$ is compact.} $\mf l$, then $\generate{\Omega}{LS}=\generate{X,\ad_{\mf h}}{LA}$ and, in particular $\overline{\reach_{\ref{eq:op-lift}}(\id)}$ is a Lie subgroup.
\end{enumerate}
\end{proposition}

\begin{proof}
\ref{it:wedges:W}:
The first part is clear since wedges are closed by definition.
Since $\generate{\Omega}{LW}$ is a wedge, it contains the linear span of all $\ad_{k_i}$, which must be contained in the edge $E(\generate{\Omega}{LW})$.
The latter is a Lie algebra, hence 
$E(\generate{\Omega}{LW})\supseteq\ad_{\mf h}$.
Since $\generate{\Omega}{LW}$ is a Lie wedge, for any $h\in\mf h$ it contains $e^{\ad_{\ad_h}} X=\Ad_{e^h}X\Ad_{e^{-h}}$.
\ref{it:wedges:LW}:
The inclusion $\supseteq$ was shown in~\ref{it:wedges:W}. 
As for the converse: let $\mf w=\generate{\ad_{\mf h},\Ad_{\mb H}^\star(X)}{wedge}$. 
It is enough to show that $\mf w$ is a Lie wedge.
First we show that $\mf w\cap E(\mf v)=\ad_{\mf h}$:
If $w\in\mf w\setminus\ad_{\mf h}$, then $w=\ad_h+\lambda Y$ where $h\in\mf h$, $\lambda>0$ and $Y\in\conv(\Ad_{\mb H}^\star(X))$. 
But since $\mf v$ is a Lie wedge, $Y\in\mf v\setminus E(\mf v)$, and $w\notin E(\mf v)$. This proves the claim.
Finally we can show that $\mf w$ is a Lie wedge.
Clearly $\mf w\subseteq \mf v$, and so $E(\mf w)\subseteq E(\mf v)$.
Hence, by the above claim, $E(\mf w)=\ad_{\mf h}$, and since $\mf w$ is invariant under the action of $\ad_{\mf h}$, it is a Lie wedge.
\ref{it:wedges:LS}:
It suffices to show that $\mf w$ is global.
For this we will use~\cite[Prop.~1.37]{HN93}. 
Since $\mf v$ and $E(\mf w)=\ad_{\mf h}$ (by~\ref{it:wedges:LW}) are global by assumption, we only need to show that $E(\mf v)\cap\mf w\subseteq E(\mf w)$. 
But this follows immediately from the claim above.
\ref{it:wedges:LA}: This is a consequence of \cite[Prop.~6.3]{Lawson99}.
\end{proof}

\subsection{Operator Lift of the Reduced Control System} \label{sec:op-lift-red}

The operator lift of~\eqref{eq:reduced}, still assuming that $X$ is linear, is defined by
\begin{align} \label{eq:reduced-lift} \tag{\sf RL}
L'(t)=X_{K(t)}L(t), \quad L(0)=\id\in\GL(\mf a)\,.
\end{align}


\begin{remark}
Although the control systems on $\mf p$ and $\mf a$ are equivalent, the same is not true for the operator lifts on $\GL(\mf p)$ and $\GL(\mf a)$. More precisely, a reachable transformation in $\GL(\mf a)$ will in general not correspond to the restriction of some reachable transformation in $\GL(\mf p)$.
In this sense the operator lift~\eqref{eq:reduced-lift} is a somewhat artificial construction which, however, turns out to be very useful.
\end{remark}

\begin{lemma} \label{lemma:reduced-wedge}
Let $X\in\mf{gl}(\mf p)$ be a linear vector field on $\mf p$.
Then it holds that $\generate{\mf X}{wedge}=\Pi_{\mf a}\circ\generate{\ad_{\mf k},\Ad_{\mb K}^\star(X)}{wedge}\circ\iota$.
\end{lemma}

\begin{proof}
This follows immediately from the definition of the induced vector fields.
\end{proof}

\section{Equivalence} \label{sec:equivalence}

The main goal of this paper is to prove that the reduced control system~\eqref{eq:reduced} on $\mf a$ is in some
sense equivalent to the original control-affine system~\eqref{eq:control-affine} on $\mf p$. 
Instead of giving a general definition of equivalence in advance, each of our main results will contain the precise sense in which the equivalence in question is to be understood.
First we will give a local equivalence result in Proposition~\ref{prop:local-equivalence}, before proving the global case. This will be separated into a projection, see Theorem~\ref{thm:projection-equivalence}, and a lift, see Theorem~\ref{thm:lift-equivalence}.




\subsection{Local equivalence}

Here we show a local equivalence result which illustrates why the definition of the reduced control system is natural. 
Note that the global equivalence result proven later does not make use of this local result.

Unless stated otherwise, we make no assumption on the smoothness or boundedness of the drift $X$.
We start with a simple but quite useful consequence of the orthogonality relation~\eqref{eq:orth-orbs}.

\begin{lemma} \label{lemma:calc}
Let $p:[0,T]\to\mf p$ be any path satisfying~\eqref{eq:control-affine} at some $t_0\in[0,T]$ and let $K\in\mb K$ be arbitrary. Then it holds that
$$
\Ad_K^{-1}\circ\,\Pi_{p(t_0)}(p'(t_0))=\Pi_{\Ad_K^{-1}(p(t_0))}\circ\Ad_K^\star(X)\circ\Ad_K^{-1}(p(t_0))\,.
$$
\end{lemma}

\begin{proof}
This is a simple computation:
\begin{align*}
\Ad_K^{-1}\circ\,\Pi_{p(t_0)}(p'(t_0))
&=
\Ad_K^{-1}\circ\,\Pi_{p(t_0)}\circ X(p(t_0))
\\&=
\Pi_{\Ad_K^{-1}(p(t_0))}\circ\Ad_K^{-1}\circ\,X(p(t_0))
\\&=
\Pi_{\Ad_K^{-1}(p(t_0))}\circ\Ad_K^\star(X)\circ\Ad_K^{-1}(p(t_0))\,,
\end{align*}
where the first equality uses~\eqref{eq:control-affine} \& \eqref{eq:orth-orbs}, the second one uses~\cite[Lem.~A.24~(iii)]{diag}, and the third one uses the definition of the pullback.
\end{proof}

The following lemma is mostly a convenient restatement of~\cite[Prop.~2.7]{diag}\footnote{
In the case of unitary diagonalization of Hermitian matrices (Example~\ref{ex:hermitian-evd}), this is a  well-known result, see~\cite[Ch.I.\S5, Thm.~1]{Rellich69}. 
Roughly speaking it states that for a differentiable path of Hermitian matrices one may choose the eigenvalue functions differentiable as well.
We extended the result to all semisimple orthogonal symmetric Lie algebras.} applied to solutions of~\eqref{eq:control-affine}. 

\begin{lemma} \label{lemma:deriv-of-projected-sol} \marginpar{better $t_0\in(0,T)$?}
Let $p:[0,T]\to\mf p$ be any path satisfying~\eqref{eq:control-affine} at some $t_0\in[0,T]$.
Then the following statements hold.\smallskip
\begin{enumerate}[(i)]
\item \label{it:diff-proj-exits} There is $a:[0,T]\to\mf a$ differentiable at $t_0$ which satisfies $\pi\circ p=\pi\circ a$. 
\item \label{it:deriv-proj-well-def} For any $b:[0,T]\to\mf a$ differentiable at $t_0$ satisfying $\pi\circ p=\pi\circ b$ there is some $w\in\mb W$ such that\,\footnote{We define $w\cdot(a(t),a'(t))=(w\cdot a(t),w\cdot a'(t))$, which naturally extends the action of $\mb W$ to the tangent bundle $T\mf a$.} $(b(t_0),b'(t_0))=w\cdot(a(t_0),a'(t_0))$, where $a$ is the object from~\ref{it:diff-proj-exits}.
\item \label{it:proj-satisfies-reduced} For any $b:[0,T]\to\mf a$ differentiable at $t_0$ satisfying $\pi\circ p=\pi\circ b$ there is some $K\in\mb K$ such that $$b(t_0)=\Ad_K^{-1}(p(t_0)) \ \text{ and }\ 
b'(t_0)=\Ad_K^{-1}\circ\,\Pi_{p(t_0)}(p'(t_0))=X_K(b(t_0))\,.$$
\item \label{it:deriv-via-diag} For any $K\in\mb K$ such that $\Ad_K^{-1}(p(t_0))\in\mf a$ and $\Ad_K^{-1}\circ\,\Pi_{p(t_0)}(p'(t_0))\in\mf a$ it holds that 
$$(\Ad_K^{-1}(p(t_0)),\Ad_K^{-1}\circ\,\Pi_{p(t_0)}(p'(t_0)))=w\cdot(a(t_0),a'(t_0))$$
for some $w\in\mb W$. Moreover, it holds that
$$
\Ad_K^{-1}\circ\,\Pi_{p(t_0)}(p'(t_0))=X_K\circ\Ad_K^{-1}(p(t_0))\,.
$$
\end{enumerate}
\end{lemma}

\begin{proof}
\ref{it:diff-proj-exits} and~\ref{it:deriv-proj-well-def} follow immediately from~\cite[Prop.~2.7]{diag}. 
For~\ref{it:proj-satisfies-reduced} we use the same proposition, together with Lemma~\ref{lemma:calc}, to obtain $b'(t_0)=\Pi_{b(t_0)}\Ad_K^\star(X)(b(t_0))$.
Since $b'(t_0)$ is diagonal by assumption, this implies that $b'(t_0)=X_K(b(t_0))$.
Finally, the first part of~\ref{it:deriv-via-diag} is a direct consequence of~\ref{it:deriv-proj-well-def}, \ref{it:proj-satisfies-reduced},
and~\cite[Coro.~A.49]{diag}. The second part follows from Lemma~\ref{lemma:calc} as before.
\end{proof}

As a converse we have the following lifting result:

\begin{lemma} \label{lemma:achievable-derivs}
Let $a_0\in\mf a$ and $K\in\mb K$ as well as $t_0\in[0,T]$ be given. 
Assume that $X$ is continuous on a neighborhood of $\Ad_K(a_0)$. 
Then there exists $p:[0,T]\to\mf p$ which solves~\eqref{eq:control-affine} on a neighborhood of $t_0$ and satisfies $p(t_0)=\Ad_K(a_0)$.
\end{lemma}

\begin{proof}
By continuity of $X$ on a neighborhood of $\Ad_K(a_0)$, Peano's Theorem \cite[Thm.~2.19]{Teschl12}, guarantees the existence of a solution $p:(t_0-\varepsilon,t_0+\varepsilon)\to\mf p$ to $p'(t)=X(p(t))$ satisfying $p(0)=\Ad_K(a_0)$. This is clearly a solution to~\eqref{eq:control-affine} at $t_0$ with all controls set to zero.
\end{proof}

These results motivate the following definition:
\begin{equation} \label{eq:tilde-lambda}
\widetilde\derv(a_0)=\{a'(t_0):\,\substack{a:[0,T]\to\mf a \text{ differentiable at some } t_0\in[0,T],\,\, a(t_0)=a_0, \text{ and } \pi\circ a=\pi\circ p,\\
 \text{ where } p:[0,T]\to\mf p \text{ satisfies } \eqref{eq:control-affine} \text{ on a neighborhood of } t_0}\, \}\,,
\end{equation}
which is the set of all possible derivatives at $a_0\in\mf a$ of solutions to~\eqref{eq:control-affine}. Now our previous results allow us to describe $\widetilde\derv(a_0)$ explicitly.

\begin{lemma} \label{lemma:tilde-lambda-explicit}
Assume that $X$ is continuous. Then it holds for every $a_0\in\mf a$ that
$$
\widetilde\derv(a_0) = \{ X_K(a_0) : K\in\mb K \text{ such that } \Pi_{a_0}\circ\Ad_K^\star(X)(a_0)\in\mf a \}.
$$
\end{lemma}

\begin{proof}
``$\subseteq$'': 
Let $a,p$ and $t_0$ be as in~\eqref{eq:tilde-lambda}. 
Lemma~\ref{lemma:deriv-of-projected-sol}~\ref{it:proj-satisfies-reduced} and its proof show that there is some $K\in\mb K$ such that
$a'(t_0)=\Pi_{a_0}\circ\Ad_K^\star(X)(a_0)$.
So $\Pi_{a_0}\circ\Ad_K^\star(X)(a_0)\in\mf a$ as desired.

``$\supseteq$'': 
By Lemma~\ref{lemma:achievable-derivs} there exists $p:[0,T]\to\mf p$ solving~\eqref{eq:control-affine} in a neighborhood of $t_0$ such that $p(t_0)=\Ad_K(a_0)$.
Lemma~\ref{lemma:calc} shows that $\Ad_K^{-1}\circ\,\Pi_{p(t_0)}(p'(t_0))=\Pi_{a_0}\circ\Ad_K^\star(X)(a_0)$.
Hence by Lemma~\ref{lemma:deriv-of-projected-sol}~\ref{it:diff-proj-exits} and~\ref{it:deriv-via-diag} there is some $a:[0,T]\to\mf a$ satisfying $\pi\circ a=\pi\circ p$ with $a(t_0)=a_0$ and $a'(t_0)=\Pi_{a_0}\circ\Ad_K^\star(X)(a_0)$, as desired.
\end{proof}

For the next result we will make use of Kostant's famous convexity theorem~\cite{Kostant73}. Recall that $\mb K_{a_0}$ and $\mb W_{a_0}$ denote the stabilizers of $a_0$ in $\mb K$ and $\mb W$, respectively.

\begin{lemma} \label{lemma:lambda-kostant}
Let $a_0\in\mf a$ and $K\in\mb K$ be given. 
Then there exists $\tilde K\in\mb K$ which satisfies $K\mb K_{a_0}=\tilde K\mb K_{a_0}$ such that $X_{\tilde K}(a_0)=\Pi_{a_0}\circ\Ad_{\tilde K}^\star(X)(a_0)\in\mf a$. 
Moreover, for any such $\tilde K$ it holds that
$\{ X_{KL}(a_0):L\in\mb K_{a_0} \}=\mathrm{conv}(\mb W_{a_0}\cdot X_{\tilde K}(a_0))$.
\end{lemma}

\begin{proof}
For the existence of $\tilde K$ let us write $\tilde K=KL$ with corresponding element $L\in\mb K_{a_0}$. 
Then, using~\cite[Lem.~A.24~(iii)]{diag} we compute
$$
\Pi_{a_0}\circ\Ad^\star_{KL}(X)(a_0)
=
\Pi_{a_0}\circ\Ad_L^{-1}\circ\Ad^\star_K(X)(a_0)
=
\Ad_L^{-1}\circ\,\Pi_{a_0}\circ\Ad^\star_K(X)(a_0)\,,
$$
and hence by~\cite[Lem.~A.45]{diag} there is some $L\in\mb K_{a_0}$ such that this expression lies in $\mf a$.
Next we compute, with $\tilde K$ as above, for arbitrary $M\in\mb K_{a_0}$
\begin{align*}
X_{\tilde KM}(a_0)
&=
\Pi_{\mf a}\circ\Ad_{\tilde KM}^\star(X)(a_0)
=
\Pi_{\mf a}\circ\Ad_M^{-1}\circ\Ad_{\tilde K}^\star(X)(a_0)
\\&=
\Pi_{\mf a}\circ\Ad_M^{-1}\circ\,\Pi_{a_0}\circ\Ad_{\tilde K}^\star(X)(a_0)
=
\Pi_{\mf a}\circ\Ad_M^{-1}\circ X_{\tilde K}(a_0)\,.
\end{align*}
The result now follows from~\cite[Lem.~A.44]{diag} and Kostant's convexity theorem.
\end{proof}


Now we are ready to prove a local equivalence result relating $\widetilde\derv$ and $\derv$, showing that they are ``almost the same''.

\begin{proposition}[Local equivalence]
\label{prop:local-equivalence}
Let $a_0\in\mf a$ and $K\in\mb K$ be arbitrary and assume that $X$ is continuous on a neighborhood of $\Ad_{K}(a_0)$. It holds that
$$
\widetilde\derv(a_0)
\subseteq
\derv(a_0)
=
\bigcup_{v\in \widetilde\derv(a_0)} \conv(\mb W_{a_0}v)
\subseteq
\conv(\widetilde\derv(a_0))
$$
so, in particular,
$
\conv(\derv(a_0))
=
\conv(\widetilde\derv(a_0))\,,
$
and if $a_0$ is regular, then it holds that 
$
\derv(a_0)=\widetilde\derv(a_0)\,.
$
Moreover, assuming that $X(0)\in\mf a$, it holds  at the origin of $\mf a$ that
$
\{X(0)\}=\widetilde\derv(0)\subseteq\derv(0)=\conv(X(0))\,.
$
\end{proposition}

\begin{proof}
The first inclusion is due to Lemma~\ref{lemma:tilde-lambda-explicit},
the first equality follows from Lemma~\ref{lemma:lambda-kostant} together with Lemma~\ref{lemma:tilde-lambda-explicit},
and the rest is straightforward.
\end{proof}

Proposition~\ref{prop:local-equivalence} tells us that the definition of the set of achievable derivatives $\derv(a)$ is ``too large'' whenever $a$ is non-regular, however only in a negligible way since the convex hulls coincide (recall the relaxation result~\cite[Ch.~2.4 Thm.~2]{Aubin84}). See Example~\ref{ex:approx-necessary} for a consequence of this fact.

\subsection{Projection}

Our main results describe the equivalence of the control-affine system~\eqref{eq:control-affine} on $\mf p$ and the reduced control system~\eqref{eq:reduced} on $\mf a$. 
The first direction is projecting from $\mf p$ to $\mf a$. 
This means that given a solution $p:[0,T]\to \mf p$ we are looking for a solution $a:[0,T]\to \mf a$ satisfying $\pi\circ a=\pi\circ p$. 
Since semisimple orthogonal symmetric Lie algebras generally correspond to some kind of matrix diagonalization, this step could also be called diagonalization.

We start with a special case in which we are given a diagonalization of $p$.

\begin{lemma} \label{lemma:special-projection}
Let $p:[0,T]\to \mf p$ be a solution of~\eqref{eq:control-affine} such that there exist differentiable functions $a:[0,T]\to \mf a$ and $K:[0,T]\to\mb K$ with $p(t)=\Ad_{K(t)}^{-1}(a(t))$. Then $a'(t)=X_{K(t)}(a(t))$ for all $t\in[0,T]$; in particular, $a$ solves~\eqref{eq:reduced} everywhere.
\end{lemma}

\begin{proof}
By differentiating and considering the part orthogonal to the orbit we obtain $a'(t)=\Ad_{K(t)}^{-1}\circ\Pi_{p(t)}(p'(t))$, see~\cite[Lem.~2.3]{diag}. Then the result follows from Lemma~\ref{lemma:deriv-of-projected-sol}~\ref{it:deriv-via-diag}.
\end{proof}

Now let us consider the general case.
The first difficulty is that $a$ is not uniquely determined. 
This will be remedied by choosing a (closed) Weyl chamber $\mf w\subset\mf a$ and requiring that $a$ take values in $\mf w$. 
A consequence of this is that we may introduce kinks where the solution hits the boundary of $\mf w$. 
Fortunately, this leaves $a$ absolutely continuous which allows us to show that $a$ satisfies the differential inclusion almost everywhere. 

\begin{theorem} \label{thm:projection-equivalence}
Let $p:[0,T]\to\mf p$ be a solution to the control system~\eqref{eq:control-affine} and let $a^\down:[0,T]\to\mf w$ be the unique path which satisfies $\pi\circ a^\down=\pi\circ p$.
Then $a^\down$ is a solution to the reduced control system~\eqref{eq:reduced} (and hence also to~\eqref{eq:relaxed}).
\end{theorem}

\begin{proof}
By~\cite[Prop.~2.1~(v)]{diag} the path $a^\down$ is absolutely continuous. 
Let $J\subseteq[0,T]$ be the subset on which both $p$ and $a^\down$ are differentiable. This set still has full (Lebesgue) measure.
For $t_0\in J$, by Lemma~\ref{lemma:deriv-of-projected-sol}~\ref{it:proj-satisfies-reduced} it holds that $(a^\down)'(t_0)=X_K(a^\down(t_0))$ for some $K\in\mb K$, which proves that $a^\down$ satisfies the differential inclusion~\eqref{eq:inclusion} almost everywhere. 
By Filippov's theorem, see~\cite[Thm.~2.3]{Smirnov02}, $a^\down$ is a solution to~\eqref{eq:reduced}.
\end{proof}
Alternatively, one can prove Thm.~\ref{thm:projection-equivalence} without making use of Filippov's theorem as follows:

By~\cite[Prop.~2.37]{diag} there exists a (Lebesgue) measurable function $K:[0,T]\to\mb K$ such that $a^\down(t)=\Ad_{K(t)}(p(t))\in\mf a$ and $(a^\down(t))'=\Ad_{K(t)}(\Pi_{p(t)}(p'(t)))\in\mf a$ almost everywhere. 
Hence by the proof of Lemma~\ref{lemma:special-projection} it holds almost everywhere that $(a^\down)'(t)=X_{K(t)}(a^\down(t))$ and so $a^\down$ solves~\eqref{eq:reduced}.\medskip

In~\cite{diag} we proved several results which show that if $p:[0,T]\to\mf p$ has a certain smoothness, then, in certain cases one can choose $a:[0,T]\to\mf a$ satisfying $\pi\circ a=\pi\circ p$ with the same smoothness. 
This allows us to strengthen the result above in some instances.

\begin{proposition}
Let $p:[0,T]\to\mf p$ be a solution to the control system~\eqref{eq:control-affine}. Then there exists $a:[0,T]\to\mf a$ satisfying $\pi\circ a=\pi\circ p$ and solving~\eqref{eq:reduced} such that:\smallskip
\begin{enumerate}[(i)]
\item \label{it:Ck-diag} if $p$ is $C^\ell$, for $\ell=1,\ldots,\infty$, and regular, then $a$ can be chosen $C^\ell$;
\item \label{it:analytic-diag} if $p$ is real analytic, then $a$ can be chosen real analytic;
\item \label{it:diff-diag} if $p$ is (continuously) differentiable, then $a$ can be chosen (continuously) differentiable.\smallskip
\end{enumerate}
Moreover, in~\ref{it:Ck-diag} and~\ref{it:analytic-diag} we can choose $a$ as before and $K:[0,T]\to\mb K$ such that $a=\Ad_{K(t)}^{-1}(p(t))$ and such that $K$ is $C^\ell$ (resp.~real analytic). Then it holds that $a'(t)=X_{K(t)}(a(t))$, i.e.~$a$ solves~\eqref{eq:reduced} with control function $K$.
\end{proposition}

\begin{proof}
Item~\ref{it:Ck-diag} follows from~\cite[Prop.~2.14]{diag}, \ref{it:analytic-diag} follows from~\cite[Thm.~2.23]{diag}, and \ref{it:diff-diag} follows from~\cite[Thm.~2.11]{diag}, in each case using Lemma~\ref{lemma:special-projection}. 
In the cases~\ref{it:Ck-diag} and~\ref{it:analytic-diag}, the same results provide $K:[0,T]\to\mb K$, and again Lemma~\ref{lemma:special-projection} shows that $a'(t)=X_{K(t)}(a(t))$. 
\end{proof}

\subsection{Lift}

The task of this section is the following: given a solution to the reduced control system~\eqref{eq:reduced} construct a solution to the original system which is, at least approximately, a lift of the former. 

For regular solutions to the reduced control system we can construct an exact lift as well as a corresponding control function $k:[0,T]\to\mf k$. In particular, if the control directions $k_1,\ldots,k_m\in\mf k$ of~\eqref{eq:control-affine} span $\mf k$ (which we called Assumption~\ref{it:full-ctrl-2}), then one easily finds the corresponding control functions $u_i$.

To properly state the result, we have to define an appropriate inverse of $\ad_p:\mf k\to\mf p$ for $p\in\mf p$. Note that the kernel of this map is exactly the commutant $\mf k_p$, and due to orthogonality the image is $\mf p_p^\perp$. Hence there is a unique inverse $\ad_p^{-1}:\mf p_p^\perp\to\mf k_p^\perp$.

\begin{proposition} \label{prop:regular-lift}
Let $X$ be $C^\ell$ and let $a:[0,T]\to\mf a$ be a solution to the reduced control system~\eqref{eq:reduced} with $C^r$ ($r\geq 1$) control function $K:[0,T]\to\mb K$ such that $a$ takes only regular values. 
If we set $p(t)=\Ad_{K(t)}(a(t))$ and define $k:[0,T]\to\mf k$ by
$$
k(t) = K'(t)K^{-1}(t) + \ad_{p(t)}^{-1}\circ\,\Pi_{p(t)}^\perp\circ X(p(t))\,,
$$
then $k$ is of class $C^{\min(\ell,r-1)}$, and $p$ satisfies $p'(t)=(\ad_{k(t)}+X)(p(t))$.
\end{proposition}
\marginpar{double check $C^{\min(\ell,r-1)}$}

\begin{proof}
By differentiating\footnote{\label{footnote:diff}If $K:I\to\mb K$ is differentiable at some $t\in I$, then
$\tfrac{d}{dt}\Ad_{K(t)} = \ad_{K'(t)K(t)^{-1}}\circ\Ad_{K(t)} = \Ad_{K(t)}\circ\ad_{K(t)^{-1}K'(t)} \,.$} 
we get that
\begin{align*}
p'(t) 
=
\ad_{K'(t)K^{-1}(t)}(p(t))+\Ad_{K(t)}(a'(t))
=
\ad_{K'(t)K^{-1}(t)}(p(t))+\Pi_{p(t)}\circ X(p(t))\,,
\end{align*}
since 
$\Ad_{K(t)}(a'(t)) 
=
\Ad_{K(t)}\circ \Pi_{\mf a}\circ\Ad_{K(t)}^{-1}\circ X\circ \Ad_{K(t)}(a(t))
=
\Pi_{p(t)}\circ X (p(t))$,
where we used that $p(t)$ is regular to introduce $\Pi_{p(t)}$. 
Hence 
\begin{align*}
(\ad_{k(t)}+X)(p(t))
&=
\ad_{K'(t)K^{-1}(t)}(p(t)) - \Pi_{p(t)}^\perp(X(p(t))) + X(p(t))
\\&=
\ad_{K'(t)K^{-1}(t)}(p(t)) + \Pi_{p(t)}(X(p(t)))=
p'(t)\,,
\end{align*}
as desired.
\end{proof}
The control Hamiltonian in Proposition~\ref{prop:regular-lift} has two components. To the \emph{induced control} $K'(t)K^{-1}(t)$, which one might naively expect to do the job, one has to add the \emph{compensating control} $\ad_{p(t)}^{-1}\circ\,\Pi_{p(t)}^\perp\circ X(p(t))$ which deals with the orbital component of $X(p(t))$, cp.~\cite[Lem.~2.3]{diag}.

In practice one might find the lift $p$ of $a$ without knowing a corresponding control function $K$ for~\eqref{eq:reduced}. In this case any $K$ diagonalizing $p$ and $\Pi_p\circ X(p)$ will do:

\begin{lemma}
Let $a:[0,T]\to\mf a$ be a regular solution to~\eqref{eq:reduced}. Assume that $p:[0,T]\to\mf p$ satisfies $p(t)=\Ad_{K(t)}(a(t))$ and $a'(t)=\Ad^{-1}_{K(t)}\circ\,\Pi_{p(t)}\circ X(p(t))$. Then $a'(t)=X_K(t)(a(t))$.
\end{lemma} 

\begin{proof}
This is straightforward using the definition of $X_K$ and regularity of $a$:
$
X_{K(t)}(a(t))
=
\Pi_{\mf a}\circ\Ad_{K(t)}^{-1}\circ X\circ\Ad_{K(t)}(a(t))
=
\Ad^{-1}_{K(t)}\circ\Pi_{p(t)}\circ X(p(t))
=
a'(t)$.
\end{proof}

If we allow for non-regular solutions, an exact lift might not even exist, as shown in the following example. 

\begin{example} \label{ex:approx-necessary}
To see that approximating solutions cannot be avoided in general, consider a system where $X(0)\neq0$. Then $p\equiv0$ is not a solution of~\eqref{eq:control-affine}, but $a\equiv0$ is a solution of $\eqref{eq:reduced}$. Indeed, by Kostant's convexity theorem, and assuming that $X(0)\in\mf a$, it follows from Proposition~\ref{prop:local-equivalence} that
\begin{align*}
\derv(0)
=
\{\Pi_{\mf a}\circ\Ad_K(X(0)):K\in\mb K\}
=
\conv(\mb W(X(0)))\,,
\end{align*}
and hence $\derv(0)$ contains the convex combination
$\frac{1}{|\mb W|}\sum_{w\in\mb  W} w\cdot X(0)$,
which equals $0$, the unique fixed point of a Weyl group action. Thus $a\equiv0$ is a solution to~\eqref{eq:reduced}.
\end{example}


\marginpar{Optional: Show that lifting is non-trivial even if we consider the differential inclusion with $\widetilde\derv$.}

For this reason, we have to look for an approximate lift in general. Before we prove the existence of such a lift, we need the following technical result.

\begin{lemma}
\label{lemma:remove-hamiltonian}
Let $\mb G$ be a Lie group and $\mb K$ be a compact subgroup such that that the norm on $\mf g$ is invariant under $\mb K$. If $\delta:[0,T]\to\mf g$ is differentiable and $\delta(0)=0$, then for every integrable $h:[0,T]\to\mf k$ it holds that 
\begin{align*}
\|\delta(t)\| \leq \int_0^t\|\ad_{h(s)}(\delta(t))+\delta'(s)\|\,ds,
\end{align*}
for all $t\in[0,T]$.
\end{lemma}

\begin{proof}
Let $\phi:[0,T]\to\mb K$ satisfy $\phi'(t)=\phi(t) h(t)$.
We compute (cf.~Footnote~\ref{footnote:diff})
\begin{align}
\|\delta(t)\|
&=\notag
\|\Ad_{\phi(t)}(\delta(t))\|
=
\Big\|\Ad_{\phi(0)}(\delta(0)) + \int_0^t\frac{d}{ds}(\Ad_{\phi(s)}(\delta(s)))\,ds\Big\|
\\&=
\Big\|\Ad_{\phi(0)}(\delta(0)) + \int_0^t
\Ad_{\phi(s)}\circ\ad_{h(s)}(\delta(s))+\Ad_{\phi(s)}(\delta'(s))
\,ds\Big\|\notag\\
&\leq \|\delta(0)\| +
\int_0^t\|
\ad_{h(s)}(\delta(s))+\delta'(s)\|\,ds\,.\notag
\end{align}
\end{proof}

\marginpar{Q: Can we add time dependence?}

Finally we can prove:

\begin{theorem}[Approximate Lifting Result] \label{thm:lift-equivalence}
Assume that $X$ is locally Lipschitz and linearly bounded\footnote{By this we mean that $\|X(v)\|\leq C_1\|v\|+C_2$ for some $C_1,C_2\geq0$.} with constants $C_1,C_2$, and let $\rm a:[0,T]\to\mf a$ be any solution to the reduced control system~\eqref{eq:reduced} with control function $K:[0,T]\to\mb K$.
Then $p:=\Ad_K(\rm a)$---which is a lift of $\rm a$ to $\mf p$---can be approximated by solutions to the original control system~\eqref{eq:control-affine} to arbitrary degree.
More precisely, for every $\epsilon>0$ there exists a solution $p_{\varepsilon}:[0,T]\to\mf p$ to~\eqref{eq:control-affine} such that $\|\Ad_K(\rm a)-p_{\varepsilon}\|_{\infty}\leq\epsilon$.
\end{theorem}

\begin{proof}
We start by proving the result under stronger assumptions, and then show that we can weaken the assumptions while maintaining uniform convergence on $[0,T]$.

First we assume that $X$ and $K$ are real analytic and that $a(0)$ is regular. Moreover, we invoke Assumption~\ref{it:full-ctrl-2}, meaning that the control directions $k_1,\ldots,k_m$ span $\mf k$.
Then the solution $a$ is also real analytic since it satisfies $a'(t)=X_{K(t)}(a(t))$ and the map $(a,K)\mapsto X_K(a)$ is real analytic.
Since the non-regular points in $\mf a$ are formed by a finite union of hyperplanes, $a$ will be regular with finitely many exceptions $t_1,\ldots,t_n$ in $[0,T]$. 
We define the set $J_\varepsilon:=[0,T]\setminus\bigcup_{i=1}^n(t_i-\epsilon,t_i+\epsilon)$,
as well as $p(t)=\Ad_{K(t)}(a(t))$
and the control function $k_\varepsilon(t)=K'(t)K^{-1}(t)+{\bf 1}_{J_\varepsilon}(t)\ad_{p(t)}^{-1}\Pi_{p(t)}^\perp(X(p(t)))$.
Note that $k_\varepsilon$ is (piecewise, in time) real-analytic and bounded.
Hence we can define $p_\varepsilon$ as the solution of $p_\varepsilon'(t)=(\ad_{k_\varepsilon(t)}+X)(p_{\varepsilon}(t))$, with $p_{\varepsilon}(0)=p(0)$. 

By Lemma~\ref{lemma:remove-hamiltonian} we find that $\|p_\varepsilon(t)\|\leq\int_0^t\|X(p_\varepsilon(s))\|ds+\|p_\varepsilon(0)\|$. 
Since $X$ is linearly bounded and using Gr{\"o}nwall's inequality\footnote{
Recall that Gr{\"o}nwall's inequality states that if $\alpha\geq0$ is non-decreasing, $\beta,u$ are continuous on $[0,T]$, and $u(t)\leq\alpha(t)+\int_0^t\beta(s)u(s)ds$ for all $t\in[0,T]$, then $u(t)\leq\alpha(t)\exp(\int_0^t\beta(s)ds)$ for all $t\in[0,T]$.
}
we obtain $\|p_\varepsilon(t)\|\leq (\|a(0)\|+tC_2)e^{tC_1}$. In particular there is some $R>0$ independent of $\varepsilon$ such that $\|p(t)\|\leq R$ and $\|p_\varepsilon(t)\|\leq R$ for all $t\in[0,T]$. Restricting to this compact domain, we may assume that $X$ is in fact globally Lipschitz with constant $L$.

Setting $\delta=p-p_{\varepsilon}$ we obtain
$\delta'(t)=\ad_{k_\varepsilon(t)}(\delta(t))-X(p_{\varepsilon}(t))-\ad_{k_\varepsilon}(t)(p(t)) + p'(t)$
and using Lemma~\ref{lemma:remove-hamiltonian} we get
$\|\delta(t)\|
\leq 
\int_0^t \|-X(p_\varepsilon(s)) - \ad_{k_\varepsilon(s)}(p(s)) + p'(s)\| ds$.
Using that
\begin{align*}
p'(t)-\ad_{k_\varepsilon(t)}(p(t))
&=
\ad_{K'(t)K^{-1}(t)}p(t) + \Pi_{\Ad_{K(t)}(\mf a)}\circ X(p(t))
\\&\quad\qquad
-\ad_{K'(t)K^{-1}(t)}p(t) +  {\bf 1}_{J_\epsilon}(t) \Pi_{p(t)}^\perp\circ X(p(t))
\\&=
{\bf 1}_{J_\epsilon}(t) X(p(t)) + {\bf 1}_{J_\epsilon^c}(t) \Pi_{\Ad_{K(t)}(\mf a)}\circ X(p(t))\,,
\end{align*}
we obtain
\begin{align*}
\|\delta(t)\|
&\leq \int_0^t
\|{\bf 1}_{J_\epsilon}(s) X(p(s))-X(p_\varepsilon(s)) 
+
{\bf 1}_{J_\epsilon^c}(s) (\Pi_{\Ad_{K(s)}(\mf a)} X(p(s)) - X(p_\varepsilon(s)))\|ds
\\&\leq
\int_0^t L\|\delta(s)\|ds + 2\mu(J_\epsilon^c)(C_1R+C_2)
\end{align*}
where $\|\cdot\|_{\infty}$ denotes the supremum norm
and $\mu$ denotes the Lebesgue measure. Finally, we again apply Gr{\"o}nwall's inequality
to obtain
$\|\delta(t)\|\leq 2\mu(J_\epsilon^c)(C_1R+C_2) e^{Lt}$
for all $t\in[0,T]$.
Since $\mu(J_\varepsilon^c)\to0$ as $\epsilon\to0$, this shows that $p_\varepsilon$ converges uniformly to $p$ on $[0,T]$.

Now we show that the result also holds under the more general assumptions. 
This will follow from a sequence of standard approximations.
Let $X$, $a$, and $K$ be as in the statement and use Assumptions~\ref{it:full-ctrl} \&~\ref{it:fast-ctrl}.
Let some $\varepsilon>0$ be given. 
Again we define $p(t)=\Ad_{K(t)}(a(t))$. 
Now let $K^{(m)}$ be a sequence of real analytic controls converging uniformly to $K$ and let $a_0^{(m)}$ be a sequence of regular points converging to $a(0)$. Let $a^{(m)}$ be the solution to~\eqref{eq:reduced} with initial point $a_0^{(m)}$ and control function $K^{(m)}$. Then by~\cite[Thm.~1]{Sontag98} the $a^{(m)}$ converge uniformly to $a$, and setting $p^{(m)}=\Ad_{K^{(m)}}(a^{(m)})$ we find that the $p^{(m)}$ converge uniformly to $p$. In particular there is $m$ such that $\|p-p^{(m)}\|_\infty\leq\tfrac\varepsilon4$.
Now let $X^{(n)}$ be a sequence of real analytic and linearly bounded 
vector fields converging uniformly on compact subsets to $X$. Let $a^{(m,n)}$ be the corresponding solutions and $p^{(m,n)}=\Ad_{K^{(m)}}(a^{(m,n)})$. 
Then by~\cite[Thm.~3.5]{Khalil02} $p^{(m,n)}\to p^{(m)}$ and there is some $n$ such that $\|p^{(m)}-p^{(m,n)}\|_\infty\leq\tfrac\varepsilon4$.
Now we can use the result proven above to find a solution $p^{(m,n)}_\varepsilon$ to~\eqref{eq:control-affine} using Assumption~\ref{it:full-ctrl-2} such that $\|p^{(m,n)}-p^{(m,n)}_\varepsilon\|_\infty\leq\tfrac\varepsilon4$.
Finally due to~\cite{Liu97} we can drop Assumption~\ref{it:full-ctrl-2} and obtain solutions $p^{(m,n,k)}_\varepsilon$ to ~\eqref{eq:control-affine} such that for some $k$ we have $\|p^{(m,n)}_\varepsilon-p^{(m,n,k)}_\varepsilon\|_\infty\leq\tfrac\varepsilon4$. Combining all these approximations then yields the result.
\end{proof}

\marginpar{TODO: Understand explosion of control near singularity}

\begin{remark}
In general the control $k_\varepsilon$ obtained by setting $\varepsilon=0$ need not be integrable since the expression $\ad^{-1}_{p(t)}$ typically leads to singularities of order $t^{-1}$ 
as $p$ passes through a non-regular point. Nevertheless, in can happen that the controls do not explode even as we pass through a non-regular point, cf., e.g,
the worked example in Section~\ref{sec:worked-example}.
\end{remark}

\section{Consequences} \label{sec:consequences}

The equivalence results proven above allow us to easily deduce several useful consequences on important control theoretic notions like reachability, stabilizability, controllability, and accessibility. Some additional basic properties are collected in Appendix~\ref{app:ctrl-sys-props} for reference.

\subsection{Speed Limit}

One of the reasons why the original control system~\eqref{eq:control-affine} is difficult to work with is the presence of unbounded controls, and the resulting fact that there are points in the state space which are far apart but can be joined in an arbitrarily short amount of time. Since these are exactly the points which are identified in the reduced control system, this cannot occur anymore. Indeed, we can define the \emph{speed limit} $c:\mf a\to\R_{\geq0}$ by
$$
c(a)=\max_{K\in\mb K} \|X_K(a)\|.
$$
Then we have the following result:

\begin{proposition} \label{prop:speed-limit} 
If the drift $X$ is continuous, then the speed limit $c$ is well-defined and continuous. In particular $c$ is bounded on bounded subsets of $\mf a$.
\end{proposition}

\begin{proof}
As $\mb K$ is compact and $K\mapsto X_K(a)$ is continuous, the image is also compact and hence $c$ is well defined.
Since all the vector fields $X_K$ are continuous, so is the map $f:\mb K\times (\overline{B_\epsilon(a_0)}\cap \mf a)\to\mathbb R$, $(K,a)\mapsto \|X_K(a)\|$ for all $a_0\in\mf a$, $\varepsilon>0$.
In particular $f$ is uniformly continuous which readily implies continuity of $a\mapsto \max_{K\in\mb K} \|X_K(a)\|=c(a)$.
\end{proof}

Given any solution $a:[0,T]\to\mf a$ to~\eqref{eq:relaxed} such that $c(a(t))\neq0$, it holds that  
$$
T
\geq\int_{a(0)}^{a(T)} \frac{\|da\|}{c(a)}
\geq \frac{\ell(a)}{\max_{t\in[0,T]}c(a(t))},
$$
where $\ell(a)$ denotes the length of $a$.

\subsection{Reachability} 

We start with the \emph{reachable set of $a_0$ at time $T$} for~\eqref{eq:reduced}. We denote
$$
\reach_{\ref{eq:reduced}}(a_0,T) := \{a(T): \,a:[0,T]\to\mf a \text{ solves } \eqref{eq:reduced},\, a(0)=a_0\}
$$
for any $T\geq0$. By
$
\reach_{\ref{eq:reduced}}(a_0) := \bigcup_{T\geq0} \reach_{\ref{eq:reduced}}(a_0,T)
$
we denote the \emph{all-time reachable set of $a_0$}, and we define the \emph{reachable set of $a_0$ up to time $T$} by
$
\reach_{\ref{eq:reduced}}(a_0,[0,T]) := \bigcup_{t\in[0,T]} \reach_{\ref{eq:reduced}}(a_0,t)
$
for any $ T\geq0$.
The definitions for the control systems (\ref{eq:control-affine}), (\ref{eq:inclusion}), (\ref{eq:relaxed}), (\ref{eq:op-lift}), and (\ref{eq:reduced-lift}) are analogous. 

\begin{remark}
Note that, although the reduced control system~\eqref{eq:reduced} is symmetric under the Weyl group action, the reachable set in general does not have the same symmetry as it depends on the initial state. 
However, due to Proposition~\ref{prop:sol-in-weyl}, if the solution starts in the Weyl chamber $\mf w$, then it holds that $\pi(\reach(a_0,T))=\pi(\reach(a_0,T)\cap\mf w)$. Together with Proposition~\ref{prop:equivalence-reach} below this shows that all relevant information concerning reachability is held in the Weyl chamber which contains the initial state.
\end{remark}

The equivalence results of Section~\ref{sec:equivalence} are formulated at the level of solutions, and they immediately imply the equivalence of reachable sets up to closure and $\mb K$-orbits.

\begin{proposition} \label{prop:equivalence-reach}
Assume that $X$ is locally Lipschitz and linearly bounded. Let $T>0$ and $p_0\in\mf p$, $a_0\in\mf a$ with $\pi(p_0)=\pi(a_0)$ be given. Then it holds that
\begin{align*}
\mathsf{reach}_{\ref{eq:control-affine}}(p_0,T)
\subseteq
\Ad_{\mb K}(\mathsf{reach}_{\ref{eq:relaxed}}(a_0,T))
\subseteq
\overline{\mathsf{reach}_{\ref{eq:control-affine}}(p_0,T)},
\end{align*}
where the reachable sets refer the the control-affine system~\eqref{eq:control-affine} on $\mf p$ and the relaxed control system~\eqref{eq:relaxed} on $\mf a$. In particular the closures coincide:
$$
\overline{\mathsf{reach}_{\ref{eq:control-affine}}(p_0,T)}
=
\Ad_{\mb K}(\overline{\mathsf{reach}_{\ref{eq:relaxed}}(a_0,T)})\,.
$$
Finally, all statements remain true is we substitute~\eqref{eq:relaxed} with~\eqref{eq:reduced}.
\end{proposition}

\begin{proof}
We prove the result only for~\eqref{eq:reduced} since the proof for~\eqref{eq:relaxed} is analogous.
First let $p:[0,T]\to\mf p$ be a solution to~\eqref{eq:control-affine}. 
By Theorem~\ref{thm:projection-equivalence} we obtain a solution $a^\down:[0,T]\to\mf w$ with $\pi(a(T))=\pi(p(T))$ to~\eqref{eq:reduced}. This proves the first inclusion.
Conversely, let $a:[0,T]\to\mf a$ be a solution to~\eqref{eq:reduced} and let $p_1,p_2\in\mf p$ be such that $\pi(p_1)=\pi(a(0))$ and $\pi(p_2)=\pi(a(T))$. 
Due to Theorem~\ref{thm:lift-equivalence} there exists for every $\varepsilon>0$ a solution $p:[0,T]\to\mf p$ to~\eqref{eq:control-affine} such that $d(\pi(a(t)),\pi(p(t)))\leq\varepsilon$ where $d$ refers to the quotient metric induced by $\pi$.
Now let some $\varepsilon>0$ be given and let $K_1,K_2\in\mb K$ be such that $\Ad_{K_1}(p(0))$ is $\varepsilon$-close to $p_1$ and such that $\Ad_{K_2}(p(T))$ is $\varepsilon$-close to $p_2$.
By approximately implementing $\Ad_{K_1}^{-1}$ on $[0,\varepsilon]$ and $\Ad_{K_2}$ on $[T-\varepsilon,T]$ we can find a solution to~\eqref{eq:control-affine} which equals $p$ on $[\varepsilon,T-\varepsilon]$ and approximately starts at $p_1$ and approximately ends at $p_2$.
Using arguments similar to those in the proof of Theorem~\ref{thm:lift-equivalence} one can deduce that $p_2\in\overline{\reach_{\ref{eq:control-affine}}(p_1,T)}$.
\marginpar{make lift more rigorous}
\end{proof}
Note that the analogous result is true for the all-time reachable sets. \medskip

Now consider $X$ linear.
We can also use the operator lift~\eqref{eq:reduced-lift} to understand reachability in the reduced system~\eqref{eq:reduced}. 
Indeed it is clear that $b\in\reach_{\ref{eq:reduced}}(a,T)$ if and only if there is some $L\in\reach_{\ref{eq:reduced-lift}}(\id,T)$ such that $La=b$. 
In fact it holds that $\overline{\reach_{\ref{eq:reduced-lift}}(\id)}$ is the Lie subsemigroup of $\GL(\mf a)$ generated by $\mf X$, see~\cite[Prop.~6.2]{Lawson99}.
Proposition~\ref{prop:speed-limit} shows that $\reach_{\ref{eq:reduced}}(a_0,[0,T])$ is bounded. 
If $X$ is Lipschitz, then Proposition~\ref{prop:lipschitz-diff-incl}~\ref{it:compact} guarantees compactness of $\reach_{\ref{eq:relaxed}}(a_0,[0,T])$, cp.~\cite[Thm.~3]{Boscain04}.

\subsection{Stabilizability} \marginpar{Reference definition?}

In practice one often wants to keep the system close to a certain state, i.e.~one wants to stabilize the state. We define the \emph{set of stabilizable states}, denoted $\stab_{\ref{eq:reduced}}$, as follows:
a point $a_0\in\mf a$ is in $\stab_{\ref{eq:reduced}}$ if for all $T>0$ and all $\varepsilon>0$ there is a solution $a:[0,T]\to\mf a$ to~\eqref{eq:reduced} with $a(0)=a_0$ and which takes values in $B_\varepsilon(a_0)$.
Moreover, we say that a point $a_0$ is \emph{strongly stabilizable}\footnote{Strongly stabilizable states are also called equilibrium states, see~\cite[p.~124]{Sontag98}.} if the constant path $a\equiv a_0$ is a solution to~\eqref{eq:reduced}.
The definition for the other control systems is analogous.
Note that we only consider open-loop controls here and that we are not using feedback.

\begin{lemma} \label{lemma:stab-inclusion}
Assume that $X$ is Lipschitz.
Given any point $a_0\in\mf a$ the following statements hold.\smallskip
\begin{enumerate}[(i)]
\item \label{it:strong-stab} $a_0$ is strongly stabilizable w.r.t~\eqref{eq:reduced} if and only if $0\in\derv(a_0)$.
\item \label{it:stab} $a_0$ is stabilizable w.r.t~\eqref{eq:reduced} if and only if $0\in\conv(\derv(a_0))$.\smallskip
\end{enumerate}
In fact these statements hold true for all continuous differential inclusions with closed values.
\end{lemma}

\begin{proof}
\ref{it:strong-stab}: If $a\equiv a_0$ is a solution to~\eqref{eq:reduced}, then $0=a'(t)\in\derv(a_0)$ for almost all $t\in[0,T]$. Conversely, if $0\in\derv(a_0)$, then $a\equiv a_0$ is a solution to~\eqref{eq:reduced}.
\ref{it:stab}: If $0\in\conv(\derv(a_0))$ then $a\equiv a_0$ is a solution to~\eqref{eq:relaxed}. By the Relaxation Theorem~\cite[Ch.~2.4, Thm.~2]{Aubin84} (which requires the Lipschitz property) the constant solution can be approximated in~\eqref{eq:reduced} and hence $a_0$ is stabilizable. If $0\notin\conv(\derv(a_0))$ there is a linear functional $\alpha$ on $\mf a$ such that $\alpha\leq-\delta$ on $\derv(a_0)$ for some $\delta>0$. By continuity we may assume that this is true for all $b\in\mf a$ in some neighborhood $B_\varepsilon(a_0)$ of $a_0$. Hence there is some time $T>0$ such that every solution to~\eqref{eq:reduced} and starting at $a_0$ must leave $B_\varepsilon(a_0)$ until time $T$.
\end{proof} 
In particular a point is stabilizable for~\eqref{eq:reduced} if and only if it is strongly stabilizable for~\eqref{eq:relaxed}, and for~\eqref{eq:relaxed} both notions coincide.\medskip

We have the following specialization of Proposition~\ref{prop:regular-lift} for strongly stabilizable states.

\begin{proposition}
The following statements hold.\smallskip
\begin{enumerate}[(i)]
\item \label{it:strong-stab-ap} If there exists $p_0=\Ad_K(a_0)\in\mf p$ as well as $k\in\mf k$ such that $(X+\ad_k)(p_0)=0$, then $a_0$ is strongly stabilizable. In fact it holds that $X_K(a_0)=0$.
\item \label{it:strong-stab-regular} Conversely, assume that $a_0$ is regular and strongly stabilizable with $X_K(a_0)=0$. Then setting $k_c=\ad_{p_0}^{-1}\circ\,\Pi_{p_0}^\perp\circ X(p_0)$ it holds that $(X+\ad_k)(p_0)=0$, where $p_0=\Ad_K(a_0)$.
\end{enumerate}
\end{proposition}

\begin{proof}
\ref{it:strong-stab-ap}: Using~\eqref{eq:orth-orbs} and~\cite[Lem.~A.24~(i)]{diag}, the assumption $(X+\ad_k)(p_0)=0$ yields $X_K(a_0)=0$ after a short computation.
\ref{it:strong-stab-regular}:
First note that for $a_0$ regular and $p_0=\Ad_K(a_0)$ it holds that $X_K(a_0)=\Ad_K^{-1}\circ\,\Pi_{p_0}\circ X(p_0)$, and in particular $X_K(a_0)=0$ if and only if $\Pi_{p_0}\circ X(p_0)=0$.
Then it just remains to plug in and compute $(X+\ad_{k_c})(p_0)=(X-\Pi_{p_0}^\perp\circ X)(p_0)=\Pi_{p_0}\circ X(p_0)=0$.
\end{proof}

Suppose that the control directions $k_1,\ldots,k_m$ in~\eqref{eq:control-affine} span the entire Lie algebra $\mf k$.
Then we can rephrase the proposition above as follows: If $p_0$ is strongly stabilizable, then so is $a_0$. Conversely, if $a_0$ is strongly stabilizable and regular, then there is a corresponding strongly stabilizable $p_0$ in the $\mb K$-orbit of $a_0$.

\subsection{Viability}

Let $R$ be a subset of $\mf a$. 
We call $R$ \emph{viable} for~\eqref{eq:reduced} if for every $a_0\in R$, there exists a solution $a:[0,\infty)\to\mf a$ to~\eqref{eq:reduced} with $a(0)=a_0$ which takes values only in $R$. 
For differential inclusions viability of closed subsets can be restated more geometrically using tangent cones, see~\cite[Thm.~5.2]{Smirnov02} as well as~\cite[Thm.~6.5.5]{Carja07} for the time-dependent version. 
Note that a point $a_0$ is strongly stabilizable if and only if $\{a_0\}$ is viable for~\eqref{eq:reduced}.

First some notation: for a set $S\subseteq\mf p$ we denote by $S^\flat\subseteq\mf a$ the set of all $a\in\mf a$ with $\pi(a)\in\pi(S)$. For a set $R\subseteq\mf a$ we denote by $R^\sharp\subseteq\mf p$ the set of all $p\in\mf p$ with $\pi(p)\in\pi(R)$.  
Note that $S^\flat$ is always $\mb W$-invariant and $R^\sharp$ is always $\mb K$-invariant.

\begin{lemma}
Let $S\subset\mf p$ be viable for~\eqref{eq:control-affine}, then $S^\flat$ is viable for~\eqref{eq:reduced}.
\end{lemma}

\begin{proof}
Let $a_0\in S^\flat$ and let $p_0\in S$ be any lift of $a_0$. By viability of $S$ there is a solution $p:[0,\infty)\to S$ and by Theorem~\ref{thm:projection-equivalence} there is a corresponding solution $a^\down:[0,\infty)\to\mf a$ with values in $S^\flat$.
\end{proof}

Due to Example~\ref{ex:approx-necessary} the converse cannot hold exactly. 
However, we have the following approximate result. 
We say that $S\subseteq\mf p$ is \emph{approximately viable} for~\eqref{eq:control-affine} if for every $p_0\in S$, every $T>0$ and every $\varepsilon$-neighborhood $U$ of $S$ there is a solution $p:[0,T]\to\mf p$ with $p(0)=p_0$ and taking values only in $U$.

\begin{proposition}
Let $R\subset\mf a$ be viable for~\eqref{eq:relaxed}.
Then $R^\sharp$ is approximately viable for~\eqref{eq:control-affine}.
\end{proposition}

\begin{proof}
Let $p_0\in S$, some $T>0$, and an $\varepsilon$-neighborhood $U$ of $R^\sharp$ be given.
Let $a_0\in R$ be such that $\pi(a_0)=\pi(p_0)$.
Since $R$ is viable, there exists some solution $a:[0,T]\to R$.
By Theorem~\ref{thm:lift-equivalence} there is some $\varepsilon$-approximate lift $p$ of $a$.
As in the proof of Proposition~\ref{prop:equivalence-reach} we may assume that $p(0)=p_0$.
\marginpar{make more rigorous}
Hence $p$ remains in $U$ and $R^\sharp$ is approximately viable.
\end{proof}

Note that even if $R$ consists of regular points, $R^\sharp$ need not be (exactly) viable.

\subsection{Invariant Subsets}

Again, let $R\subseteq\mf a$ be any subset. We say that $R$ is \emph{invariant} for~\eqref{eq:control-affine} if there does not exist a solution $a:[0,T]\to\mf a$ with $a(0)\in S$ and $a(T)\notin S$. For differential inclusions, invariance of closed subsets can be characterized using a tangent cone condition, cf.~\cite[Thm.~5.6]{Smirnov02}.

\begin{proposition}
Let $S\subseteq\mf p$ be a closed, $\mb K$-invariant subset. Then $S$ is invariant with respect to~\eqref{eq:control-affine} if and only if $S^\flat$ is invariant with respect to~\eqref{eq:reduced} (or, equivalently, \eqref{eq:relaxed}).
\end{proposition}

\begin{proof}
Note that $S$ is invariant if and only if for every $p_0\in S$ and $T>0$ it holds that $\reach_{\ref{eq:control-affine}}(p_0,T)\subseteq S$.
Hence the result follows immediately from Proposition~\ref{prop:equivalence-reach}.
\end{proof}

Let $S\subseteq\R^n$ be an arbitrary subset. 
We define by $\GL(n;S)$ the set of invertible linear maps on $\R^n$ which map $S$ into itself. 
This set always contains the identity and is closed under composition. 
We call such sets \emph{semigroups}.\footnote{This terminology is established in the literature on Lie semigroups. Elsewhere such objects are often called monoids.}

The following lemma, essentially a restatement of~\cite[Prop.~1]{OSID23}, shows how the properties of $S$ can affect those of $\GL(n;S)$.

\begin{lemma} \label{lemma:trafo-semigroup}
Let $S\subseteq\R^n$ be an arbitrary subset. Then the following hold.\smallskip
\begin{enumerate}[(i)]
\item \label{it:semi-closed} If $S$ is closed, then $\GL(n;S)$ is closed\footnote{It may happen, however, that $\GL(n;S)$ is not closed in $\mf{gl}(n)$, i.e.~there might be non-invertible limit points.} in $\GL(n)$.
\item \label{it:semi-bounded} If $S$ is bounded and $\linspan(S)=\R^n$, then $\GL(n;S)$ is bounded.
\end{enumerate}
\end{lemma}


In the following we will always assume that $S$ is closed.
Let $\mf{gl}(n;S)$ denote the Lie wedge of $\GL(n;S)$.
Recall the definition of the Bouligand contingent cone, cf.~Definition~\ref{def:contingent-cone}. 

\begin{lemma} \label{lemma:inv-vf}
Let $S\subseteq\R^n$ be closed and $X\in\mf{gl}(n)$ be linear vector field. 
Then $X\in\mf{gl}(n;S)$ if and only if $X(x)\in {T}_x^-S$ for all $x\in S$.
\end{lemma}

\begin{proof}
Consider $e^{tX}$ for $t\geq0$. If $e^{tX}(S)\subseteq S$ then clearly $X(x)=\frac{d}{dt}\big|_{t=0}e^{tX}(x)\in T_x^-S$. 
The converse follows from~\cite[Thm.~5.6]{Smirnov02}.
\end{proof}

\marginpar{Def: (cite?) a set is regular if Clarke and Bouligand coincide (e.g., convex sets, mfds)}

For closed $S$, Lemma~\ref{lemma:trafo-semigroup}~\ref{it:semi-closed} shows that $\mf{gl}(n;S)$ is a global Lie wedge (cf.~\cite{Lawson99}). 
The generated Lie semigroup is denoted $\mathrm{MGL}(n;S)$.

\begin{lemma} \label{lemma:emb-submfd}
If $S$ is a closed embedded submanifold of $\R^n$, then $\mathrm{MGL}(n,S)$ is a Lie group.
\end{lemma}

\begin{proof}
Since $S$ is embedded, the contingent cone at any point of $S$ equals its tangent space.
Hence by Lemma~\ref{lemma:inv-vf}, $\mf{gl}(n;S)$ consists of those linear vector fields which are tangent to $S$.
In particular, if $X\in\mf{gl}(n;S)$, then so is $-X$. Hence $\mf{gl}(n;S)$ is a Lie algebra and $\mathrm{MGL}(n,S)$ is a Lie group.
\end{proof}

\marginpar{Non-example: Open disk, reversed vector fields are not complete, inv requires closed set}

\subsection{Accessibility}

Systems with irreversible behavior---such as physical systems with dissipation---typically are not controllable (we will discuss controllability below). 
The next best property one can hope for is accessibility. 
Recall that $\mf X=\{X_K:K\in\mb K\}$ denotes the set of induced vector fields. 
Let $\generate{\mf X}{LA}$ denote the generated Lie algebra, and $(\generate{\mf X}{LA})_{a_0}=\{Y(a_0):Y\in\generate{\mf X}{LA}\}$ the evaluation at $a_0\in\mf a$.
Similarly we denote by $\generate{\mf X}{LS}$ the Lie saturate of $\mf X$.

Let a closed embedded submanifold $R\subseteq\mf a$ be invariant for~\eqref{eq:reduced}. 
If $(\generate{\mf X}{LA})_{a_0}={T}_{a_0}R$ for all $a_0\in R$, we say that $\mf X$ is \emph{completely nonholonomic on $R$} or \emph{bracket generating on $R$}, see~\cite[Def.~5.10]{Agrachev04}, or that $\mf X$ satisfies the \emph{accessibility rank condition}~\cite[Def.~4.3.2]{Sontag98}.
Note that if $X$ is linear, then $\mf X\subset\mf{gl}(\mf a)$. In particular $\generate{\mf X}{LA}$ is finite dimensional. 

The system~\eqref{eq:reduced} is \emph{accessible at $a_0$ on $R$} if $\reach_{\ref{eq:reduced}}(a_0,[0,T])$ has non-empty interior in $R$ for all $T>0$. The accessibility rank condition implies accessibility, see~\cite[Thm.~9]{Sontag98}. 

If $X$ is analytic, the orbit of $a_0$ is an immersed submanifold $O_{a_0}$ and the system satisfies the accessibility rank condition on $O_{a_0}$. This is part of the Nagano--Sussmann Orbit Theorem, see~\cite[Thm.~5.1]{Agrachev04}. For this reason, in many cases, it is not a restriction to assume accessibility.

In the differential inclusion picture we can define a stronger notion. We say that~\eqref{eq:reduced} is \emph{directly accessible at $a_0$ on $R$} if $\linspan(X_K(a_0):K\in\mb K)={T}_{a_0}R$. This means linear combinations suffice to generate the entire tangent space without the use of Lie brackets. Conveniently, this property is relatively easy to check by considering the differential inclusion. Note that direct accessibility in~\eqref{eq:reduced} is equivalent to direct accessibility in~\eqref{eq:relaxed}.

\begin{remark}
Direct accessibility is a useful property:
for certain cost functions, non-linear control problems can be described using sub-Riemannian geometry~\cite{Agrachev19} which are notoriously difficult problems. 
If, however, the system is directly accessible, the problem becomes Riemannian, thus simplifying considerably.
\end{remark}

\begin{proposition}
Assume that $X$ is linear. Let $\mb T\subseteq\GL(\mf a;R)$ be a Lie subgroup with Lie algebra $\mf t$, and assume $\mf X\subseteq\mf t$. 
If $\mb T$ acts locally transitively\footnote{If $\phi_{a_0}:T\to R$ denotes the Lie group action at $a_0$, then $\mb T$ acts locally transitively at $a_0$ if $D\phi_{a_0}(\id):\mf t\to{T}_{a_0}R$ is surjective.} at $a_0$ on $R$, then\smallskip
\begin{enumerate}[(i)]
\item if $\linspan(\mf X)=\mf t$, then~\eqref{eq:reduced} is directly accessible at $a_0$ on $R$;
\item if $\generate{\mf X}{LA}=\mf t$, then~\eqref{eq:reduced} is accessible at $a_0$ on $R$.
\end{enumerate}
\end{proposition}

\begin{proof}
Note that the differential of the action, $D\phi_{a_0}(e):\mf t\to{T}_{a_0}S$, which is surjective by assumption, is exactly the evaluation of the corresponding vector field at $a_0$. Now the result follows directly from the definitions.
\end{proof}

\begin{proposition}
Let $X$ be real analytic and assume that there is some $a_0\in\mf a$ such that~\eqref{eq:reduced} is directly accessible at $a_0$.
Then~\eqref{eq:reduced} is directly accessible on an open dense subset of $\mf a$ whose complement has measure zero.
In particular this happens if $\mf g=\mf k\oplus\mf p$ is simple and $X(0)\neq0$. 
\end{proposition}

\begin{proof}
Choose a set of induced vector fields $X_{K_i}\in\mf X$ for $i=1,\ldots,n$ such that the $X_{K_i}(a_0)$ form a basis of ${T}_{a_0}\mf a$. 
Now consider the determinant of these vector fields $a_0\mapsto\det(X_{K_1}(a_0),\ldots,X_{K_n}(a_0))$ as a function on $\mf a$. 
By assumption this is a real analytic function on $\mf a$ which does not vanish at $a_0$.
Hence it is non-zero on an open dense set whose complement has measure zero,
and clearly~\eqref{eq:reduced} is directly accessible whenever the function is non-zero.

Now assume that $\mf g=\mf k\oplus\mf p$ is simple and $X(0)\neq0$.
Let $K\in\mb K$ be such that $\Ad_K^\star(X)(0)\in\mf a$ and so $X_K(0)\neq0$.
Since the Weyl group acts irreducibly on $\mf a$, and since, as in Example~\ref{ex:approx-necessary}, it holds that $\derv(0)=\conv(\mb W X_K(0))$, we see that $0\in\interior(\derv(0))$ and so~\eqref{eq:reduced} is directly accessible at $0$.
\end{proof}

\begin{proposition}
Let $R\subseteq\mf a$ be a $\mb W$-invariant closed embedded submanifold and assume that the reduced system~\eqref{eq:reduced} is directly accessible at some regular $a_0$ on $R$. 
Then~\eqref{eq:control-affine} is accessible on $R^\sharp$ at every $p_0$ with $\pi(p_0)=\pi(a_0)$.
\end{proposition}

\begin{proof} \marginpar{extend to non-regular $a$, more general $R$}
The Lie algebra corresponding to~\eqref{eq:control-affine} is $\generate{X+\ad_{\mf k}}{LA}=\generate{X,\ad_{\mf k}}{LA}$. 
Since every Lie algebra is invariant under its adjoint action, it holds that $\Ad_K^{-1}\circ X\circ\Ad_K\in\generate{X,\ad_{\mf k}}{LA}$ for all $K\in\mb K$.
For the same reason we can also assume that $p_0=a_0$.
The tangent space at $p_0$ takes the form $T_{p_0}R^\sharp=T_{p_0}R\oplus \ad_{\mf k}(p_0)$.
The assumption means that $\linspan(X_K(a_0):K\in\mb K)=T_{a_0}R$, and so every element in $T_{p_0}R^\sharp$ is a linear combination of some $\Ad_K^{-1}\circ X\circ\Ad_K(p_0)$ with $K\in\mb K$ and some $\ad_k(p_0)$ with $k\in\mf k$.
This concludes the proof.
\end{proof}

\subsection{Controllability}

Let $R$ be an invariant subset for~\eqref{eq:reduced}.
Then we say that~\eqref{eq:reduced} is \emph{controllable on $R$} if for every $a_0\in R$ it holds that $\reach_{\ref{eq:reduced}}(a_0)=R$. 
We say that~\eqref{eq:reduced} is \emph{controllable on $R$ in time $T$} if for every $a_0\in R$ it holds that $\reach_{\ref{eq:reduced}}(a_0,[0,T])=R$, 
see~\cite[Ch.~3]{Sontag98}. 
We define \emph{approximate controllablility} analogously except that we consider the closure of the reachable set.
Then the following is an immediate consequence of Proposition~\ref{prop:equivalence-reach}.

\begin{proposition}
Assume that $X$ is locally Lipschitz and linearly bounded.
Let $S\subseteq\mf p$ be $\mb K$-invariant and invariant for~\eqref{eq:control-affine}. The following statements hold.\smallskip
\begin{enumerate}[(i)]
\item If~\eqref{eq:control-affine} is (approximately) controllable on $S$, then~\eqref{eq:reduced} is (approximately) controllable on $S^\flat$.
\item If~\eqref{eq:reduced} is approximately controllable on $S^\flat$, then~\eqref{eq:control-affine} is approximately controllable on $S$.
\end{enumerate} \smallskip
All statements remain true it we consider (approximate) controllability in time $T$.
\end{proposition}
%

Let $a_0\in R$. We say that~\eqref{eq:reduced} is \emph{locally controllable at $a_0$ on $R$} if $\reach_{\ref{eq:reduced}}(a_0,[0,T])$ contains an open neighborhood of $a_0$ (in the subspace topology of $R$) for all $T>0$. Moreover we say that~\eqref{eq:reduced} is \emph{locally directly controllable at $a_0$ on $R$} if $0\in\interior(\derv(a_0))$ (where the interior is taken in the topology of $T_{a_0}R$).

\begin{lemma} \label{lemma:local-ctrl}
Assume that $X$ is continuous.
If~\eqref{eq:reduced} or~\eqref{eq:relaxed} is locally directly controllable at $a_0$ on $R$, then it is locally controllable at $a_0$ on $R$.
\end{lemma}

\begin{proof}
Assume that~\eqref{eq:reduced} or~\eqref{eq:relaxed} is locally directly controllable at $a_0$ on $S$. By continuity of $X$, there is a neighborhood $A$ of $a_0$ in $S$ and some $\varepsilon>0$ such that $B_\varepsilon(0)\subset\derv(a)$ for all $a\in A$. Hence every path $b:[0,T]\to A$ is a solution if $\|b'(t)\|\leq\varepsilon$, and in particular this implies local controllability at $a_0$ on $S$.
\end{proof}

\begin{proposition} \label{prop:controllable}
Each of the following statements is sufficient for controllability on $R$ (where we necessarily assume the latter to be path-connected):\smallskip
\begin{enumerate}[(i)]
\item \label{it:lift-trans} $\reach_{\ref{eq:reduced-lift}}(\id)$ acts transitively on $R$;
\item \label{it:group-arc} $\reach_{\ref{eq:reduced-lift}}(\id)$ is a group and~\eqref{eq:reduced} satisfies the accessibility rank condition everywhere on $R$;
\item \label{it:locally-contr-everywhere} \eqref{eq:reduced} is directly locally controllable everywhere on $R$ and $R$ is a closed embedded submanifold (without boundary).
\end{enumerate}
\end{proposition}

\begin{proof}
\ref{it:lift-trans}: Obvious. 
\ref{it:group-arc}: Follows from~\cite[Coro.~4.3.12]{Sontag98}.
\ref{it:locally-contr-everywhere}: 
For any initial state $a_0\in R$, Lemma~\ref{lemma:local-ctrl} shows that the reachable set $\reach_{\ref{eq:reduced}}(a_0)$ is open. 
Considering the negated differential inclusion, which by the proof of Lemma~\ref{lemma:emb-submfd} is still invariant, shows that the set of points that are not reachable from $a_0$ (in the original system) is also open. As $R$ is connected, and $a_0$ is clearly reachable, the system is controllable.
\end{proof}

Note that the Weyl group $\mb W$ acts on the Lie algebra $\mf{gl}(\mf a)$ by Lie algebra automorphisms and, by Lemma~\ref{lemma:weyl-symmetry-induced-vf}, the set $\mf X$ of induced vector fields is invariant under this action.

\begin{proposition}
Assume that $X$ is linear. 
Let $\mb T\subseteq\GL(\mf a;R)$ be a connected Lie subgroup with Lie algebra $\mf t$ and assume that $\emptyset\neq\mf X\subseteq\mf t$. 
If $\mb T$ acts transitively on $R$, and $\mb W$ acts irreducibly on $\mf t$, then~\eqref{eq:reduced} is controllable on $R$.
\end{proposition}

\begin{proof}
Due to irreducibility, it holds that $\linspan(\mf X)=\mf t$.
In particular $\reach_{\ref{eq:reduced-lift}}(\id)=\mb T$ and the result follows from Proposition~\ref{prop:controllable}~\ref{it:lift-trans}.
\end{proof}

\section{Simulation and Weyl Order} \label{sec:simulation}

In this section we assume that $X$ is an affine linear vector field on $\mf p$. 
Using the Weyl group action we obtain a preorder on $\mf a$ which acts as a kind of resource, allowing one system to simulate another. 

Let $a\in\mf a$. We define the \emph{Weyl polytope} of $a$ via $P(a):=\conv(\mb Wa)$, that is, $P(a)$ the convex hull of the Weyl group orbit of $a$. Since $\mb W\subseteq\mf{gl}(\mf a)$, we can consider the convex hull of $\mb W$ in $\mf{gl}(\mf a)$. 
For $a,b\in\mf a$, it is clear that $a\in P(b)$ if and only if there is some $\overline w\in\conv(\mb W)$ such that $a=\overline wb$.
It is easy to show that $\conv(\mb W)$ is a semigroup. 
Hence we can define a preorder, called \emph{majorization}, on $\mf a$ by declaring for $a,b\in\mf a$ that $a\preceq b :\iff a\in P(b)$.
Indeed, reflexivity is clear and transitivity follows immediately from the fact that $\conv(\mb W)$ is a semigroup. 

\begin{lemma}
The set of vertices of $P(a)$ is exactly $\mb Wa$. In particular if $P(a)=P(b)$, then $\mb Wa=\mb Wb$.
\end{lemma}

\begin{proof}
By definition the set of vertices is a subset of the Weyl group orbit of $a$. However, since $P(a)$ is invariant under $\mb W$, and since $\mb W$ acts transitively on the orbit, all elements of the orbit must be vertices.
\end{proof}

Note that if $a\preceq b$ and $b\preceq a$, then $P(a)=P(b)$.
Hence $a$ and $b$ belong to the same Weyl group orbit so $\preceq$ induces a partial order on the orbits (or, equivalently, in a closed Weyl chamber).


\marginpar{Desired: Any linear monotone map lies in $\conv(W)$.}

The following continuity property will be useful later.

\begin{lemma} \label{lemma:majorization-lipschitz}
The set-valued map $P:\mf a\to\cP(\mf a)$ defined by $a\mapsto\conv(\mb Wa)$ is Lipschitz continuous with Lipschitz constant $1$.
\end{lemma}

\begin{proof}
Let $\mb W=\{w_i:i=1,\ldots m\}$ with $m$ the order of the Weyl group and let $\Delta^{m-1}$ denote the standard simplex. Consider the map $f:\mf a\times\Delta^{m-1}\to\mf a$ given by $(a,\lambda)\mapsto \sum_{i=1}^m \lambda_i w_i\cdot a$. This map is clearly $1$-Lipschitz in $a$, and by~\cite[Prop.~2.4]{Smirnov02} it holds that $a\mapsto P(a)=f(a,\Delta^{m-1})$ is $1$-Lipschitz as well.
\end{proof}

The main result of this section is the following:

\begin{theorem}[Simulation]
\label{thm:simulation}
Assume that $X$ is affine linear. Let $a:[0,\infty)\to\mf a$ be a solution to the relaxed control system~\eqref{eq:relaxed} and let $b_0\in\mf a$ such that $a(0)=a_0\preceq b_0$. Then there exists a solution $b:[0,\infty)\to\mf a$ to~\eqref{eq:relaxed} with $b(0)=b_0$ such that $a(t)\preceq b(t)$ for all $t\in[0,\infty)$.
\end{theorem}

\begin{proof}
First we prove the result with the additional assumption that $a$ is differentiable. Then by Proposition~\ref{prop:sol-in-weyl}, $a^\down$ is also a solution. Since it is continuous, \cite[Lem.~B.5~(iii)]{diag} shows that $a^\down$ is right-differentiable. 

Consider the set-valued maps $A,C:[0,\infty)\to\mc P(\mf a)$ defined by $A(t)=\{x\in\mf a:x\succeq a(t)\}$ and $C(t)=A(t)\cap\mf w$.
The main idea is to show that for each $t\in[0,\infty)$ and $x\in C(t)$ there is some $v\in\derv(x)$ such that $v\in T_x\mf w$ and $v-a'(t)\in T_xA(t)$.
Intuitively this means that for every point majorizing $a(t)$, there exists a derivative preserving majorization and the Weyl chamber for an infinitesimal amount of time.

By assumption, $a(t)\in {\rm relint}(F)$ for some face $F$ of $\conv(\mb Wx)$. By Result~\ref{res:faces-of-orbitopes} there is some $\Omega\subset\mf w$ such that $F=\conv(\mb W_\Omega x)$ 
For some enumeration $w_i$, with $i=1,\ldots,k$, of $\mb W_\Omega$ and some $\lambda\in\Delta^{k-1}$ it holds that
$a(t)=\sum_{i=1}^k \lambda_i w_i\cdot x$.
Hence using affine linearity of $X$ and Lemma~\ref{lemma:weyl-symmetry-induced-vf} we compute
$$
a'(t)
=\sum_{j} \mu_j X_{K_j}\big(\sum_i\lambda_i w_i x\big)
=\sum_{i,j} \lambda_i\mu_j w_i X_{K_j N_i} (x)
=\sum_{i}w_i \sum_j\lambda_i\mu_j X_{K_j N_i} (x)
$$
where $N_i\in\mb K$ is any representative of $w_i$.
Now consider the achievable derivative
$$
v= \sum_{i,j}\lambda_i\mu_j X_{K_j N_i}(x) \in \conv(\derv(x))\,,
$$
then
$$
v-a'(t)=\sum_{i}(\id-w_i)\sum_j\lambda_i\mu_j X_{K_j N_i} x,
$$
which lies in the affine span of $F$, which coincides with the tangent space $T_{a(t)}F$. 
This shows that $v-a'(t)\in T_xA(t)$.
Moreover there exists some $w\in\mb W_x$ such that $\overline v=w\cdot v\in T_x\mf w$. 
But then we still have $\overline v-a'(t) = (w\cdot v-v)+(v-a'(t)) \in T_xA(t)$.

To show that existence of $\bar v$ implies existence of the desired solution $b$, we employ a sequence of rather technical results detailed in Appendix~\ref{app:simulation}.
It follows from Corollary~\ref{coro:weyl-polar} that $T_{a(t)}A(t)$ is the negative dual of $\mf w$.
Thus we can in fact apply Lemma~\ref{lemma:dual-geodesics} and Proposition~\ref{prop:viability-ABC}, showing that $\overline v\in \D C(t,c)(1)$, and together with Result~\ref{res:viability-time-dep} this tells us that there exists a solution $b$ to the relaxed control system~\eqref{eq:relaxed} such that $b(0)=b_0$ and such that $b(t)\in C(t)$ for all $t\geq0$, or equivalently, $b(t)\preceq a(t)$ and $b(t)\in\mf w$. This concludes the proof in the differentiable case.

Now we drop the assumption that $a$ is differentiable. 
By~\cite[Thm.~1]{Sontag98}  and~\cite[Ch.~2.4 Thm.~2]{Aubin84} there exists a sequence $a_n$ of differentiable solutions to the relaxed control system converging uniformly to $a$ on compact time intervals. 
By the above, there exist solutions $b_n$ satisfying $a_n(t)\preceq b_n(t)$ for all $t$. 
By compactness of solution set on compact time interval, cf.~Proposition~\ref{prop:lipschitz-diff-incl}~\ref{it:compact}, there is a uniformly converging subsequence $b_{k_n}$ with limit $b$. 
Since $a_n(t)\to a(t)$ it holds that $a(t)\preceq b(t)$ by~\cite[Prop.~2.1]{Smirnov02} since $b\mapsto\conv(\mb W b)$ is upper semi-continuous, see Lemma~\ref{lemma:majorization-lipschitz}, with closed values.
\end{proof}

\marginpar{Corollary: In the reduced control system this should still work in some approximate sense}

\marginpar{Corollary: In the bilinear control system:}

\marginpar{Find counterexample to reversed result.}

\marginpar{Optimal derivatives for $\derv(a)$ convex polytope. (e.g., toy model $\derv(\mb W)$)}

\section{Worked Example} \label{sec:worked-example}

We now revisit the motivating example given in the introduction in order to apply to it the theory we have developed. We will consider the following vector field $X$ on the disk $D=\{(y,z)\in\R^2:y^2+z^2\leq1\}$:
\begin{align*}
X(y,z)=(-\Gamma y,-\gamma(z-1))
\end{align*}
where $\Gamma,\gamma>0$. This system is ubiquitous in quantum mechanics since it describes the relaxation of a two-level system under the Bloch equations. The corresponding control system has been studied in~\cite{Lapert10} using the Pontryagin maximum principle. In this example we will show how the same control system can be studied using our reduction method.

By rescaling it suffices to consider $\gamma=1$. Moreover, to ensure that the flow does not leave the disk we have to require $\Gamma\geq\tfrac\gamma2=\tfrac12$. In fact we will consider $\Gamma\geq\tfrac32$ in the following to simplify the exposition. (All figures use the value $\Gamma=3$.)

We already stated that this problem can be described using the symmetric Lie algebra given in Example~\ref{ex:polar-dec}. Equivalently it can be obtained using~\cite[Ex.~1.1]{diag}.
The reduced control system is defined on the set $[-1,1]$, which can be seen as the intersection of the disk with $z$-axis\footnote{Any axis would do as they are equivalent under rotation, but the $z$-axis is special due to the symmetry of the vector field (and the physical origin of the example).}. Since the map $(a,\phi)\mapsto X_\phi(a)$ is continuous, the values of the set-valued map $\derv$ are compact intervals. In order to understand $\derv$, it suffices to find the upper envelope $u(a):=\max(\derv(a))$. One can show that
\begin{align} \label{eq:envelope}
u(a)=\begin{cases}
-\Big(\frac{1}{4(\Gamma-1)a}+\Gamma a\Big) & a\leq a_0:=\frac{-1}{2(\Gamma-1)} \\
1-a & a\geq a_0\,.
\end{cases}
\end{align}

We consider the optimal control problems of moving from the boundary of the disk to the center and vice-versa. In the reduced control system this is equivalent to moving from $-1$ to $1$. The form of~\eqref{eq:envelope} shows that this is indeed possible, but it takes infinite time to reach $1$. The optimal solution is then given by the differential equation
$a'(t)=u(a(t)), \quad a(0)=-1$,
which can be solved explicitly, and one obtains
$$
a^\star(t)=\begin{cases}
-\frac{\sqrt{-1+(1-2\Gamma)^2e^{-2\Gamma t}}}{2\sqrt{\Gamma(\Gamma-1)}} 
& t\leq t_0:=\frac{\log((\Gamma-1)(2\Gamma-1))}{2\Gamma}  \\
1-\frac{2\Gamma-1}{2(\Gamma-1)}((\Gamma-1)(2\Gamma-1))^{\frac{1}{2\Gamma}} e^{-t}
& t\geq t_0\,.
\end{cases}
$$

The next step is to lift the optimal solution $a^\star$ to the original control system to obtain a solution $p^\star$ on the disk. This solution will start on the boundary of the disk, pass through the center, and again approach the boundary of the disk. Once this optimal solution is found, we will determine the corresponding control function $\omega^\star$.

Above we have determined the upper envelope $u(a)$ of $\derv(a)$. More precisely, one can show that $u(a)=X_{\phi^\star}(a)$ where 
$$
\phi^\star(a)=
\begin{cases}
\arccos\Big(\frac{1}{2(\Gamma-1)a}\Big)+\frac\pi2 & a\leq a_0 \\
\frac\pi2 & a\geq a_0\,.
\end{cases}
$$
Thus the optimal path $p^\star(t)$ in the disk (in polar coordinates) is $(a^\star(t), \phi^\star(a^\star(t)))$,
cf.~Figure~\ref{fig:phi-p-opt}.

\begin{figure}
\centering
\includegraphics[width=0.35\textwidth]{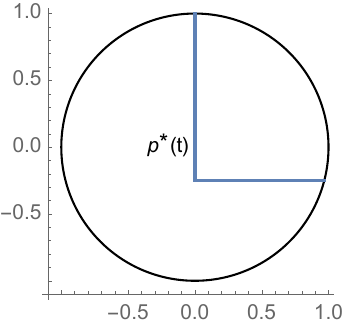}
\hspace{5mm}
\includegraphics[width=0.55\textwidth]{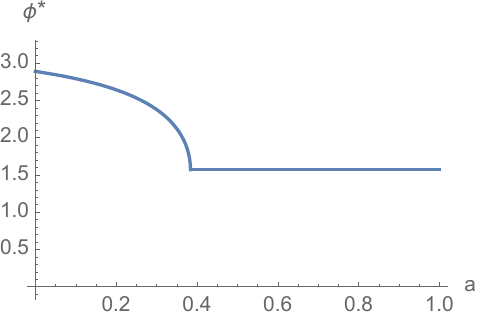}
\caption{(Left) Optimal path $p^\star(t)$ on the unit disk. The horizontal part satisfies $z=a_0=\frac{-1}{2(\Gamma-1)}$. (Right) Optimal angle $\phi^\star$ as a function of radius $a$.}
\label{fig:phi-p-opt}
\end{figure}

Finally, it remains to calculate the control function $\omega^\star$ which generates the optimal solution $p^\star(t)$.
There are two components,  $\omega_0$ and the compensation term $\omega_c$:
$$
\omega^\star = \omega_0 + \omega_c = \frac{d}{dt}\phi^\star(a^\star(t)) 
+ \ad_{p(t)}^{-1}\circ\Pi^\perp_{p(t)}\circ X(p(t))\,.
$$
Explicitly we obtain for $t<t_0$:
\begin{align*}
\omega_0(t) = -\Gamma\frac{\delta(t)+1}{\delta(t)}\sqrt{\frac{\eta}{\delta(t)-\eta}}, \quad
\omega_c(t) = -\frac{\Gamma}{\delta(t)}\sqrt{\frac{\delta(t)-\eta}{\delta(t)}}
\end{align*}
where $\eta=\frac{\Gamma}{\Gamma-1}$ and $\delta(t)=(1-2\Gamma)^2e^{-2\Gamma t}-1$.
For $t\geq t_0$ it holds that $\omega_0(t)=\omega_c(t)=0$. The optimal controls are plotted in Figure~\ref{fig:opt-ctrl}.

\begin{figure}[htb] \marginpar{integrable?}
\centering
\includegraphics[width=0.55\textwidth]{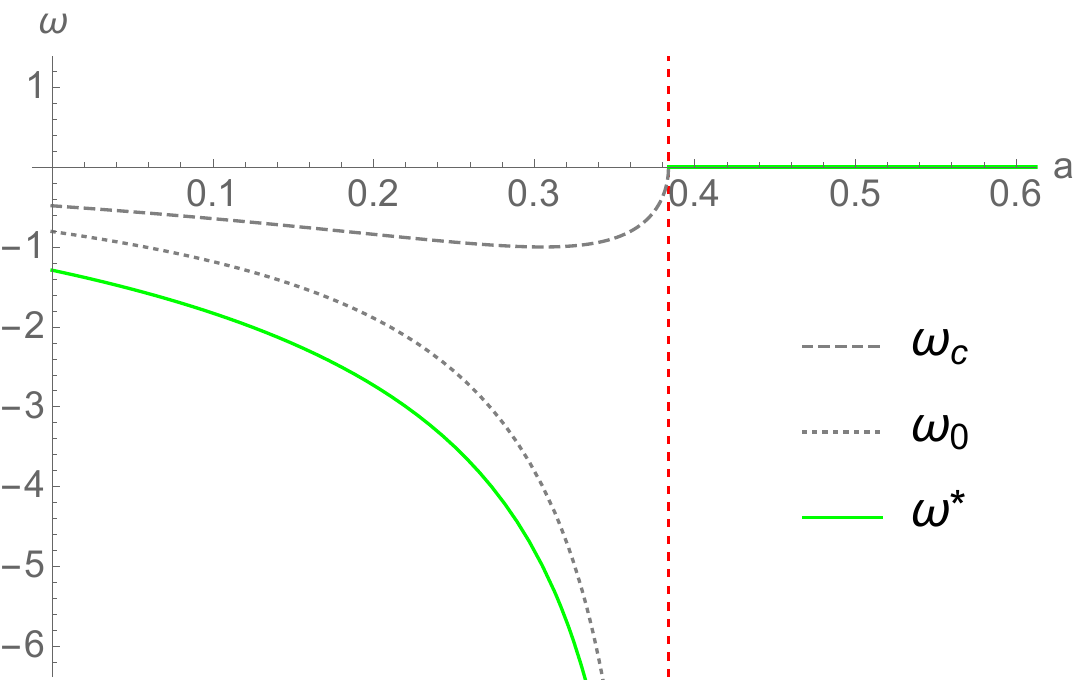}
\caption{Optimal controls for the worked example.}
\label{fig:opt-ctrl}
\end{figure}

Since $\delta(t_0)=\eta$, we see that $\lim_{t\to t_0}\omega_0(t)=-\infty$ and $\lim_{t\to t_0}\omega_c(t)=0$. In particular $\omega_0$ explodes at $t_0$ whereas $\omega_c$ is continuous.
Note also that the controls are smooth when $a^\star(t)=0$, so in this example the orbifold singularity does not pose any problems.

\section*{Acknowledgments}

The project was funded i.a.~by the Excellence Network of Bavaria under ExQM, by {\it Munich Quantum Valley} of the Bavarian State Government with funds from Hightech Agenda {\it Bayern Plus} (E.M., F.v.E. and T.S.H.), as well as the Einstein Foundation (Einstein Research Unit on quantum devices) and the {\sc math}+ Cluster of Excellence (F.v.E.).

\appendix

\section{Basic Properties of the Reduced Control Systems} \label{app:ctrl-sys-props}

Here we give some basic properties of the control systems defined in Section~\ref{sec:control-systems}. 

\subsection{Weyl Symmetry}

We start with some symmetry considerations.\footnote{ 
Recall that the Weyl group is defined as $\mb W=N_{\mb K}(\mf a)/Z_{\mb K}(\mf a)$ where $N_{\mb K}(\mf a)$ denotes the normalizer of $\mf a$ in $\mb K$ and $Z_{\mb K}(\mf a)$ the centralizer.} 

\begin{lemma} \label{lemma:pi-a-equiv}
If $w=NZ_{\mb K}(\mf a)\in\mb W$ with $N\in N_{\mb K}(\mf a)$, then it holds that $\Pi_{\mf a}\circ\Ad_N=w\circ\Pi_{\mf a}$.
\end{lemma}

\begin{proof}
First note that for $x\in\mf p$ and $K\in\mb K$ it holds that $\Pi_{\Ad_K(\mf a)}\circ\Ad_K(x) = \Ad_K\circ\,\Pi_{\mf a}(x)$, see for instance~\cite[Lem.~A.24~(iii)]{diag}.
Since $N\in N_{\mb K}(\mf a)$ we have that 
$w\cdot\Pi_{\mf a}(x)
= \Ad_N\circ\,\Pi_{\mf a}(x)
= \Pi_{\mf a}\circ\Ad_N(x)$, as desired.
\end{proof}

\begin{lemma} \label{lemma:weyl-symmetry-induced-vf}
Let $N\in N_{\mb K}(\mf a)$ and $w = N Z_{\mb K}(\mf a) \in W$ and $K\in \mb K$, then it holds that 
$X_{KN} = w^{-1}\circ X_K\circ w$, 
and hence for all $a\in\mf a$ we get 
$\derv(w\cdot a)=w\cdot \derv(a)$.
\end{lemma}

\begin{proof}
We compute using Lemma~\ref{lemma:pi-a-equiv}:
\begin{align*}
X_{KN} 
&= \Pi_{\mf a} \circ \Ad_{KN}^\star(X) \circ\iota
= \Pi_{\mf a} \circ \Ad_N^\star(\Ad_K^\star(X)) \circ\iota\\
&= \Pi_{\mf a} \circ \Ad_N^{-1} \circ \Ad_K^\star(X) \circ \Ad_N \circ\,\iota
\\&= w^{-1} \circ \Pi_{\mf a} \circ \Ad_K^\star(X) \circ\iota\circ w
= w^{-1} \circ X_K \circ w\,.
\end{align*}
In particular $X_K\circ w=w\circ X_{NK}$, which also shows the second claim.
\end{proof}

\begin{lemma}
Let $\mb K_X = \{K\in\mb K : \Ad_K^\star X=X\}$. Then $X_{SK} = X_K$ for all $S\in\mb K_X$ and $K\in\mb K$. 
\end{lemma}

\begin{proof}
As $\Ad_{SK}^\star = (\Ad_S \circ \Ad_K)^\star = \Ad_K^\star \circ \Ad_S^\star$,  we have
$X_{SK} = \Pi_{\mf a}\circ\Ad_{SK}^\star(X)\circ\iota = \Pi_{\mf a}\circ\Ad_K^\star(X)\circ\iota = X_K$.
\end{proof}

\begin{proposition} \label{prop:sol-in-weyl}
If $a:[0,\infty)\to\mf a$ is a solution to~\eqref{eq:inclusion} or \eqref{eq:relaxed}, then the unique $a^\down:[0,\infty)\to\mf w$ satisfying $\pi\circ a=\pi\circ a^\down$ is also a solution.
\end{proposition}

\begin{proof}
By~\cite[Prop.~2.1~(v)]{diag}, $a^\down$ is still absolutely continuous. 
Assume that $a$ and $a^\down$ are both differentiable at $t$. 
By~\cite[Lemma~B.5~(i)]{diag} there is some $w\in\mb W$ such that $a^\down(t)=w\cdot a(t)$ and $(a^\down)'(t)=w\cdot a'(t)$. 
By Lemma~\ref{lemma:weyl-symmetry-induced-vf} we get, using $N\in N_{\mb K}(\mf a)$ with $w = N Z_{\mb K}(\mf a)$:
$$
(a^\down)'(t)=w\cdot a'(t)=w X_K a(t)= X_{KN^{-1}} wa(t)=X_{KN^{-1}}a^\down(t)
$$
and so $a^\down$ satisfies the differential inclusion at $t$, and hence almost everywhere.
\end{proof}




\subsection{Continuity and Compactness}

\begin{lemma}
If $X$ is Lipschitz, then the set-valued function $\derv$ is also Lipschitz. This means that for all $x,y\in\mf a$,
$$
\derv(x) \subseteq \derv(y)+L\|x-y\|B_1
$$
for some (global) Lipschitz constant $L>0$ and where $B_1$ denotes the closed unit ball in $\mf a$.
\end{lemma}

This implies some convenient properties of the relaxed control system~\eqref{eq:relaxed}, see~\cite[Ch.~4]{Smirnov02}. Here we denote the set of solutions $a:[0,T]\to\mf a$ to~\eqref{eq:relaxed} with $a(0)=a_0$ by $\sol_{\ref{eq:relaxed}}(a_0,[0,T])$.

\begin{proposition} \label{prop:lipschitz-diff-incl}
Let $X$ be Lipschitz and let $a_0\in\fa$. The following holds.\smallskip
\begin{enumerate}[(i)]
\item \label{it:path-connected} The set solutions $\sol_{\ref{eq:relaxed}}(a_0,[0,T])$ is path-connected in the AC-topology\footnote{By $AC([0,T],\mf a)$ we denote the Banach space of absolutely continuous functions $a:[0,T]\to\mf a$ equipped with the norm $\|a\|_{AC}=|a(0)|+\int_0^T|a'(t)|dt$.}.
\item \label{it:compact} If $\derv$ is bounded, then $\sol_{\ref{eq:relaxed}}(a_0,[0,T])$ is compact in the standard C-topology of uniform convergence.
\item \label{it:boundary} If $a\in\sol_{\ref{eq:relaxed}}(a_0,[0,T])$ is a solution to~\eqref{eq:relaxed} with $a(T)\in\partial\reach_{\ref{eq:relaxed}}(a_0,[0,T])$, then $a(t)\in\partial\reach_{\ref{eq:relaxed}}(a_0,[0,t])$ for all $t\in[0,T]$.
\item \label{it:optimality} If $\derv$ is bounded, then there exist time-optimal solutions to~\eqref{eq:relaxed} starting in a given compact set and ending in a given closed set, assuming any such solution exists in the first place.
\item \label{it:lipschitz} If $X$ is Lipschitz with Lipschitz constant $L$, then the map $\mf a\to \cP(AC([0,T],\mf a))$ given by $a_0\mapsto\sol_{\ref{eq:relaxed}}(a_0,[0,T])$ is Lipschitz with Lipschitz constant $1+TLe^{TL}$.
\end{enumerate}
\end{proposition}


\section{Technical Results for the Simulation Theorem} \label{app:simulation}


We recall some basic facts from convex analysis and prove some technical results needed for the proof of Theorem~\ref{thm:simulation}. Our main reference is~\cite{Smirnov02}.

We start with the concept of a tangent cone to a convex set at a certain point.

\begin{definition}[Tangent cone]
Let $X$ be a normed space and let $C\subseteq X$ be a convex subset. Given any $x\in C$, the \emph{tangent cone to $C$ at $x$} is defined by
$$
T_xC = \overline{\bigcup_{\lambda>0}\frac{C-x}{\lambda}}
=\{v\in X:\lim_{\lambda\to0^+} \,d(x+\lambda v,A)/\lambda=0 \}.
$$
\end{definition}

Intuitively, $T_xC$ is the closure of the set of all directions which lie in $C$ for some small enough distance. As soon as non-convex sets come into play, the situation becomes more complicated. 

\begin{definition}[Bouligand contingent cone] \label{def:contingent-cone}
Let $X$ be a normed space and let $A\subseteq X$ be any subset. Given any $x\in A$, the \emph{contingent cone to $A$ at $x$} is defined by
$$
T_x^-A = \{v\in X:\liminf_{\lambda\to0^+} \,d(x+\lambda v,A)/\lambda=0 \}
$$
\end{definition}

The contingent cone is indeed a closed cone and for convex sets it coincides with the tangent cone, cf.~\cite[p.~38]{Smirnov02}.

A common geometric way to think of the derivative of a function $f$ in standard calculus is as a tangent space to the graph $\Gamma_f$ of the function at a given point. Using the Bouligand contingent cone we can define a derivative for set-valued function in much the same way, cf.~\cite[p.~41]{Smirnov02}.

\begin{definition}[Contingent derivative]
Let $X,Y$ be normed spaces, $F:X\to\mc P(Y)$, and $(x_0,y_0)\in\Gamma_F$. The set-valued map $D^-F(x_0,y_0):X\to\mc P(Y)$ defined by
\begin{align*}
\Gamma_{D^-F(x_0,y_0)}=T^-_{(x_0,y_0)}\Gamma_F
\end{align*}
is called the \emph{contingent derivative of $F$ at $(x_0,y_0)$}.
\end{definition}

\begin{example}
Let $f(x)=x\sin(1/x)$ (with $f(0)=0$). Then $f$ is continuous, but not differentiable at $0$. The contingent derivative is $D^-F(0,0)(x)=[-|x|,|x|]$. In particular $D^-F(0,0)(1)=[-1,1]$.
\end{example}

\begin{result}[Thm.~6.5.5 in \cite{Carja07}] \label{res:viability-time-dep}
Assume that the set-valued map $C:\R\to\R^n$ has a closed graph and the set-valued map $F:\Gamma_C\to\R^n$ is upper semi-continuous and has closed, convex values. Then the following statements are equivalent.\smallskip
\begin{enumerate}[(i)]
\item For any point $(t_0,x_0)\in\Gamma_C$ there is a solution $x:[t_0,\infty]$ to the differential inclusion $x(t)\in F(t,x(t))$ with $x(t_0)=x_0$.
\item For any $(t,x)\in\Gamma_C$ it holds that $F(t,x)\cap D^-C(t,x)(1)\neq\emptyset$.
\end{enumerate}
\end{result}

\begin{lemma} \label{lemma:hausdorff-triangle}
Let $X$ be a metric space and consider $\cP(X)$ with the Hausdorff distance $d$. Then it holds that 
$
d(x,A)\leq d(x,B)+d(B,A)
$.
\end{lemma}

Recall that for a cone $C\subseteq\R^n$, the dual cone of $C$ is defined as $C^*=\{x'\in\R^n: \braket{x',x}\geq0 \,\forall x\in C\}$. Note that if $C\subseteq D$ for two cones, then $C^*\supseteq D^*$. Moreover, for a convex set $A$ and some $x\in A$, the \emph{normal cone} of $A$ at $x$ is defined by $N_xA:=-(T_xA)^*$, i.e.~the negative of the dual of the tangent cone. The next result follows for instance from~\cite[Sec.~11.2]{Clarke13}.

\begin{lemma} \label{lemma:length-min-normal-cone}
Let a closed convex set $A\subseteq\R^n$ and a point $x\in\R^n$ be given.
If $\bar x\in A$ denotes the closest point in $A$ to $x$, then $x-\bar x\in N_{\bar x}A$.
\end{lemma}


\begin{lemma} \label{lemma:dual-geodesics}
Let $C\subseteq\R^n$ be a closed convex cone and let $x\in C$.
Then for any $y\in C$ it holds that $d(y,C^*+x)=d(y,(C^*+x)\cap C)$.
\end{lemma}

\begin{proof}
Since $C^*$ is closed and convex, there exists unique $\overline y\in C^*+x$ such that $d(y,C^*+x)=d(y,\overline y)$.
We will show that $\overline y\in C$. 
Let $v=y-\overline y$. 
By Lemma~\ref{lemma:length-min-normal-cone} it holds that $v\in N_{\overline y}(C^*+x)$. 
Since for any closed convex set $A$ and $a\in A$ we have $T_aA\supseteq A$, we compute
$$
v\in N_{\overline y}(C^*+x) = N_{\overline y-x}(C^*) = -(T_{\overline y-x}(C^*))^* \subseteq -(C^*)^* = -C,
$$
so $-v\in C$ and hence $\overline y=y-v\in C$, as desired.
\end{proof}

\begin{proposition} \label{prop:viability-ABC}
Let $I\subseteq\R$ be an open interval and let $\tilde A,\tilde B\subseteq\R^n$ be subsets.
Assume that $\tilde B$ is a closed, convex, polyhedral cone.
Let $a:I\to\tilde B$ be right differentiable and define the set-valued functions $A(t)=\tilde A+a(t)$, and $C(t)=A(t)\cap \tilde B$.
Assume that for all $b\in\tilde B$ and $t\in I$ it holds that $d(b,A(t))=d(b,C(t))$.
Let $t\in I$ and $c\in C(t)$, and assume that there is $v$ such that $v\in T_c\tilde B$ and $v-a_+'(t)\in T_{c}A(t)$. Then $v\in \mathrm D^-C(t,c)(1)$.
\end{proposition}

\begin{proof} \marginpar{write for general $t$}
We assume that $t=0$.
By definition, $v\in\D^-C(0,c)(1)$ if and only if $(1,v)\in\Gamma_{\D^-C(0,c)}=T^-_{(0,c)}\Gamma_C$. So we have to show that
$$
\liminf_{\varepsilon\to 0^+}\frac{1}{\varepsilon} d((\varepsilon,c+\varepsilon v),\Gamma_C)=0\,.
$$
In fact it is easy to see that
$d((\varepsilon,c+\varepsilon v),\Gamma_C)\leq d(c+\varepsilon v,C(\varepsilon))$.
For $\varepsilon$ small enough, $c+\varepsilon v\in\tilde B$ since $\tilde B$ is a convex polyhedron and $v\in T_c\tilde B$. So, for any $x\in\tilde B$ we have by assumption $d(c+\varepsilon v,C(\varepsilon))=d(c+\varepsilon v,A(\varepsilon))$.
Moreover using Lemma~\ref{lemma:hausdorff-triangle} we find
\begin{align*}
d(c+\varepsilon v,A(\varepsilon))
&= 
d(c+\varepsilon v,A(0) + a(\varepsilon) - a(0))
\\&\leq
d(c+\varepsilon v,A(0) + \varepsilon a_+'(0)) + d(A(0) + a(\varepsilon) - a(0),A(0) + \varepsilon a_+'(0))
\\&\leq
d(c+\varepsilon v,A(0) + \varepsilon a_+'(0)) + |a(\varepsilon)-a(0) - \varepsilon a'_+(0)|\,.
\end{align*}
Combining the results above and the assumption that $v-a_+'(0) \in T_{c}A(0)$ we see that
$$
\liminf_{\varepsilon\to 0^+} \frac{1}{\varepsilon} d((\varepsilon,c+\varepsilon v),\Gamma_C)
\leq 
\liminf_{\varepsilon\to 0^+} \frac{1}{\varepsilon}d(c+\varepsilon v,A(0) + \varepsilon a_+'(0)) + \Big|\frac{a(\varepsilon) - a(0)}{\varepsilon}-a'_+(0)\Big|
=0
$$
which concludes the proof.
\end{proof}

Note that although $T_x(A\cap B)\subseteq T_xA\cap T_xB$, the converse need not hold, which complicates the proof above. 
\marginpar{See also the following counter-example. (\todo{compare~\cite[Thm.~2.4.1]{Carja07}}).}




The following result is a restatement of \cite[Thm.~4.1]{McCarthy03}.

\begin{result} \label{res:faces-of-orbitopes}
Let $\mb W$ be a Coxeter group acting on a real, $n$-dimensional vector space $V$, and let $\mf w$ be a (closed) Weyl chamber.
Let $F$ be a codimension-$k$ face of the orbitope $\conv(\mb Wx)$ for some $x\in\mf w$. 
Then there exists a set $\Omega$ of $k$ fundamental weights belonging to the same Weyl chamber $\mf w$ such that 
$F=\conv(\mb W_\Omega x)$.
\end{result}

\begin{corollary} \label{coro:weyl-polar}
Fix a Weyl chamber $\mf w$, and let $x\in\mf w$ be a regular point. Then $T_x(\conv(\mb Wx))=-\mf w^*$, i.e.~the negative dual cone of $\mf w$.
\end{corollary}

\begin{proof} \marginpar{details}
Applying Result~\ref{res:faces-of-orbitopes} with $k=1$ we see that the fundamental weights $\omega_i$ of $\mf w$ are exactly the outward normals of the facets of $T_x(\conv(\mb Wx))$. Hence $v\in T_x(\conv(\mb Wx))$ if and only if $\braket{\omega,v}\leq0$, and since the fundamental weights generate the Weyl chamber $\mf w$, it holds that $T_x(\conv(\mb Wx))=-\mf w^*$.
\end{proof}

\bibliographystyle{spmpsci}
\bibliography{../control28vFeb21}

\end{document}